\documentclass[a4paper, 11pt]{amsart}

\pdfpageheight\paperheight 

\usepackage[english]{babel} 
\usepackage[utf8]{inputenc} 
\usepackage[T1]{fontenc}
\usepackage[a4paper, margin=1.2in]{geometry}
\usepackage{xcolor}
\usepackage[pdftex]{hyperref}
\hypersetup{
    colorlinks,
    linkcolor={red!50!black},
    citecolor={blue!50!black},
    urlcolor={blue!80!black},
   }
\usepackage[textsize=tiny]{todonotes}
\usepackage{tikz,pgfplots}
\usepackage{graphicx}
\usepackage{subfig}
\usepackage{float}
\usepackage{lmodern,amsmath,amssymb,graphicx,amsthm,mathrsfs,mathtools}
\usepackage{enumitem}
\usepackage{varwidth}
\usepackage{tasks}
\usepackage{nicematrix}
\usepackage{amsthm}

\mathtoolsset{showonlyrefs}

\DeclareMathOperator{\rank}{rank}
\DeclareMathOperator{\spn}{span}

\newcommand{\R}{\mathbb{R}}
\newcommand{\N}{\mathbb{N}}

\newcommand{\h}{\mathbf{h}}

\newcommand{\Vect}{\operatorname{Vec}}

\newcommand{\tnr}{\operatorname{TNR}}

\newcommand{\scal}[2]{\langle #1, #2 \rangle}

\newtheorem{theorem}{Theorem}[section]
\newtheorem{lemma}[theorem]{Lemma}
\newtheorem{proposition}[theorem]{Proposition}
\newtheorem{corollary}[theorem]{Corollary}

\theoremstyle{theorem}
\newtheorem*{question}{Question}

\theoremstyle{definition}
\newtheorem{definition}[theorem]{Definition}

\newtheorem*{notation}{Notation}

\theoremstyle{remark}
\newtheorem{remark}[theorem]{Remark}

\usepackage[numbers]{natbib}

\begin{document}

\title[Branching of regular abnormal geodesics]{Branching of regular abnormal geodesics in sub-Riemannian manifolds of rank two}

\author{Luca Rizzi}
\address[Rizzi]{SISSA, via Bonomea, 265 - 34136 Trieste - Italy}
\email{\href{mailto:lrizzi@sissa.it}{lrizzi@sissa.it}}

\author{Mario Sigalotti}
\address[Sigalotti]{Sorbonne Université, Université Paris Cité, CNRS, Inria, Laboratoire Jacques Louis-Lions, Paris, France}
\email{\href{mailto:mario.sigalotti@inria.fr}{mario.sigalotti@inria.fr}}

\author{Nicolò Tedesco}
\address[Tedesco]{SISSA, via Bonomea, 265 - 34136 Trieste - Italy and Sorbonne Université, Inria, Université de Paris, Laboratoire Jacques-Louis Lions, Paris-France}
\email{\href{mailto:ntedesco@sissa.it}{ntedesco@sissa.it}}


\date{\today}

\begin{abstract}
 We study branching of regular abnormal geodesics in sub-Riemannian manifolds of rank two. As an application, we provide examples of two new phenomena: continuous families of strictly abnormal branching geodesics, and geodesics with an accumulation of countably many branching points.
\end{abstract}

\maketitle

\tableofcontents

\section{Introduction}
Recently, branching of geodesics has attracted growing attention, driven in part by the difficulties it introduces in optimal transport problems within metric spaces and spaces with synthetic curvature bounds (see \cite{Cav-Hues, Cav-Mon, gigli, Rajala-Sturm} and references therein).
Although there exist many classes of non-branching spaces, including Riemannian manifolds, it is possible to find examples of branching metric spaces,  such as graphs or $\R^n$ with non-strictly convex norms \cite{BBI}, sub-Finsler Lie groups \cite{donne2025} and sub-Riemannian manifolds, as we now discuss. The first example of branching geodesics in a sub-Riemannian manifold was given in \cite{Mietton-Rizzi}, where the authors proved that normal geodesics can branch if the structure is not real-analytic, providing a characterization of branching of normal geodesics. In particular, the branching occurs in families parametrized smoothly by finitely many parameters.
This excludes branching of finite or countable sets of geodesics. In \cite{Mietton-Rizzi}, the problem was also raised of finding examples of branching of strictly abnormal geodesics. A first answer to this problem was given in  \cite{rossi2026interior}, where the authors provided an example of branching into two 
(non-smooth) strictly abnormal geodesics in a real-analytic sub-Riemannian manifold.

In this paper, we study a relevant class of abnormal geodesics in rank-two sub-Rieman\-nian manifolds, called \emph{regular} \cite{liu-sus} or \emph{nice} \cite{ABB}. The core features of regular abnormal geodesics are that, analogously to normal ones, they are projections of integral lines of a Hamiltonian vector field, and that their  sufficiently short  arcs are length-minimizing \cite{liu-sus,ABB,Agr-Sar95}. 
 This permits us to adapt the arguments of \cite{Mietton-Rizzi}, obtaining a characterization of their branching times (Theorem~\ref{th:abn_br}). A consequence of our analysis is that branching of a regular abnormal geodesic occurs in \emph{continuous} families. This is then a distinct phenomenon from the \emph{discrete} branching of abnormal geodesics observed in \cite{rossi2026interior}, consequence of branching of the corresponding lifts, which is not possible for regular abnormal ones. Another related remark is that in the real-analytic case branching of regular abnormal curves cannot occur, as is the case for normal geodesics.

The characterization of branching times for regular abnormal geodesics given in Theorem~\ref{th:abn_br} is formally analogous to the one given in \cite{Mietton-Rizzi} for normal geodesics. Namely, we prove that a regular abnormal geodesic $\gamma:[0,T]\to M$ has a regular abnormal branching at a time $\tau$ if and only if the dimension of the space of initial covectors corresponding to abnormal lifts of $\gamma|_{[0,t]}$ loses dimension at $t=\tau$. When $\gamma$ is both normal and regular abnormal, this implies that branching into normal geodesics and into regular abnormal ones occur at the same times. 

In addition, we provide criteria to establish whether the regular abnormal branching geodesics are strictly abnormal, giving two examples that are different in nature. The first example  (Section~\ref{sec:exp1}) consists of a family of geodesics in which the common subsegment (i.e. the subsegment where they coincide) is itself strictly abnormal. In the second example (Section~\ref{sec:exp2}), we consider a \emph{reference} geodesic that is both normal and regular abnormal, that splits at the same time into two distinct families of branching geodesics that are, respectively, normal and regular abnormal, proving that the regular abnormal ones are also strictly abnormal.    

In Section \ref{sect:local_branching}, we give a more general notion of branching of geodesics. In this broader setting, we provide examples of geodesics with an arbitrarily large number of branchings, which can even accumulate at some point. This is a novel phenomenon, which was forbidden in the previous frameworks. 

In Section~\ref{sect:opt_arcs} we  prove a stability 
property for the length minimality of short arcs of 
regular abnormal geodesics, generalizing the classical result of \cite{liu-sus}. As an application, we prove that if a regular abnormal geodesic has a regular abnormal branching at time $\tau$, for any \emph{compact} family of branching geodesics there exists a common time interval $[\tau-\rho, \tau+\rho]$ such that the restrictions to such an interval of all the geodesics in that family are length minimizers. \\

\noindent
\textbf{Acknowledgments.}
This project has received funding from the European Research Council (ERC) under the European Union's Horizon 2020 research and innovation programme (grant agreement GEOSUB, No.\ 945655). 
The authors acknowledge the support of INdAM, and of the IRP project GEOSUBMAN by INSMI-CNRS. We also thank A.~Agrachev and D.~Barilari for helpful discussions.

\section{Preliminaries}
\label{sec:prel}
We introduce the setting and notations, referring to \cite{ABB} for details. Let $M$ be a smooth manifold of dimension $n$ and denote by $\Vect(M)$ the space of smooth vector fields on $M$. A sub-Riemannian structure on $M$ is a pair $(\Delta,g)$, where $\Delta$ is a totally nonholonomic smooth distribution of constant rank and $g$ is a Riemannian metric on $\Delta$. One can find a \emph{generating family} $\{X_1,\dots,X_m\} \subset \Vect(M)$, tangent to $\Delta$, such that 
\begin{equation}
    g_q(v,v)= \min \left \{ \sum_{i=1}^m u_i^2 \ \,\Bigg|\,\sum_{i=1}^m u_iX_i(q) = v\right \} \quad \forall v \in \Delta(q), \, q\in M.
\end{equation}

A \emph{horizontal curve} on $M$ is an absolutely continuous curve $\gamma:[0,T] \to M$ such that $\dot{\gamma}(t) \in \Delta(\gamma(t))$ for a.e. $t\in [0,T]$ and there exists $u\in L^2([0,T],\R^m)$ verifying
\begin{equation}
    \dot{\gamma}(t) = \sum_{i=1}^m u_i(t) X_i(\gamma(t)) \quad \text{for a.e. } t\in[0,T].
\end{equation}
The control $u$ may not be unique, but there exists a unique \emph{minimal} control that verifies
\begin{equation}
\label{eq:min_cont}
    g_{\gamma(t)}(\dot{\gamma}(t), \dot{\gamma}(t)) = |u(t)|^2 \quad \text{for a.e. } t \in [0,T].
\end{equation}
The length of a horizontal curve is defined as
\begin{equation}
    L(\gamma) = \int_0^T \sqrt{g_{\gamma(t)}(\dot{\gamma}(t),\dot{\gamma}(t))} \  dt.
\end{equation}
 Any horizontal curve is the reparametrization of a unit speed one, namely 
the one whose minimal control satisfies $|u(t)| = 1$ almost everywhere. The \emph{Carnot--Carathéodory} distance between $p,q\in M$ is defined as
\begin{equation}
    d(p,q)= \inf \ \{L(\gamma) \mid \gamma(0)=p, \ \gamma(T)=q, \ \gamma \ \text{horizontal} \}.
\end{equation}
A horizontal curve 
$\gamma:[0,T] \to M$ is a \emph{length minimizer} if it verifies $L(\gamma)=d(\gamma(0),\gamma(T))$, while it is called \emph{geodesic} if it is parametrized by constant speed and short arcs of $\gamma$ are length minimizers. This means that for every $t \in [0,T]$ there exists $\varepsilon>0$ such that $\gamma|_{[t-\varepsilon, t+\varepsilon] \cap [0,T]}$ is a length minimizer.

The \emph{energy} of a horizontal curve is defined as
\begin{equation}
    J(\gamma) = \frac{1}{2} \int_0^T g_{\gamma(t)}(\dot{\gamma}(t),\dot{\gamma}(t)) \  dt.
\end{equation}
A horizontal curve is an energy minimizer if and only if it is a length minimizer parametrized with constant speed. Let us now introduce some notation.

\begin{notation}
For any $X \in \Vect(M)$ we define the smooth function $h_X: T^*M \to \R$ as $h_X(\lambda)=\scal{\lambda}{X(\pi(\lambda))}$, where $\pi:T^*M \to M$ is the bundle projection and $\scal{\cdot}{\cdot}$ denotes the pairing between a covector and a tangent vector. Given a generating family $\{X_1,\dots,X_m\}$ as above, for all $i_1,\dots,i_k \in \{1,\dots,m\}, \ k \in \N$, we denote: 
\begin{equation}
    X_{i_1\dots i_k} = [X_{i_1},\dots,[X_{i_{k-1}},X_{i_k}]\dots], \quad  h_{i_1 \dots i_k}= h_{X_{i_1\dots i_k}}.
\end{equation}
Finally, for any smooth function $h:T^*M \to \R$ we denote its \emph{Hamiltonian vector field} as $\vec{h} \in \Vect(T^*M)$, which is defined by the equation $\omega(\cdot,\vec{h})=dh$, where $\omega$ is the canonical symplectic form on $T^*M$. 
\end{notation}

We recall necessary conditions for length minimizers. Fix $q \in M$, $T>0$. There exists an open set $\mathcal{U} \subset L^2([0,T], \R^m)$ such that, for all $u \in \mathcal{U}$, a solution to the Cauchy problem
\begin{equation}
    \dot{\gamma}_u(t) = \sum_{i=1}^m u_i(t) X_i(\gamma_u(t)), \quad \gamma_u(0)=q
\end{equation}
is defined on $[0,T]$. 

\begin{definition}
    The \emph{endpoint map} at $q\in M$ is the map $E_q:\mathcal{U} \to M$, $E_q(u) = \gamma_u(T)$.  
\end{definition}

If a curve $\gamma$ is a length minimizer parametrized with constant speed (equivalently, $\gamma$ is an energy minimizer) and $u$ is its minimal control, then by \eqref{eq:min_cont} $u$ is a minimizer of the $L^2$ norm on the level set $\{E_{\gamma(0)}(u)=\gamma(T)\}$. By the Lagrange multipliers rule, there exist $\lambda_T \in T_{\gamma(T)}^*M$ and $\nu \in \{0,1\}$ with $(\lambda_T,\nu) \neq 0$ such that
\begin{equation}
\label{eq:lagr}
    \lambda_T \circ D_uE_{\gamma(0)}(\cdot)= \nu \scal{u}{\cdot}. 
\end{equation}

By propagating $\lambda_T$ over $\gamma$, we obtain a lift $\lambda:[0,T] \to M$ of $\gamma$, that we call \emph{normal} $(N)$ if $\nu=1$ and \emph{abnormal} $(A)$ if $\nu=0$. More precisely, $\lambda(\cdot)$ is defined by
\begin{equation}\label{eq:extremallift}
    \lambda(t)=(P^u_{t,T})^*\lambda_T \quad \forall t \in [0,T],
\end{equation}
where $P^u_{t_1,t_2}$ denotes flow associated with the non-autonomous vector field $X_u:= u_1X_1 + \cdots+u_mX_m$ between times $t_1$ and $t_2$, and the star denotes the pull-back. Notice that the same path can admit multiple lifts, some verifying $(N)$ and some others $(A)$. 

In case $(N)$ it turns out that $u_i(t)=h_i(\lambda(t))$ for $i\in \{1,\dots,m\}$ and a.e.\ $t \in [0,T]$, and the lift is an integral line of the \emph{sub-Riemannian Hamiltonian}
\begin{equation}
    H(\lambda) := \frac{1}{2}\sum_{i=1}^m h_i(\lambda)^2,
\end{equation}
while in case $(A)$ the lift is simply a solution of \eqref{eq:extremallift}, namely
\begin{equation}
    \label{eq:pmp}
        \dot{\lambda}(t) = \sum_{i=1}^m u_i(t) \vec{h}_i(\lambda(t)) \quad \text{for a.e.} \ t \in [0,T],
\end{equation}
verifying also $h_i(\lambda(t))=0$ for $i\in \{1,\dots,m\}$ and a.e. $t \in [0,T]$.

\begin{definition}
    A horizontal path is called a \emph{normal} (resp.\ \emph{abnormal}) \emph{path} if it admits a normal (resp.\ abnormal) lift.
\end{definition}

\subsection{Branching of normal geodesics}
\label{sect:norm_branch}

We recall some results from \cite{Mietton-Rizzi}.

\begin{definition}
\label{def:n_br}
    We say that a normal geodesic $\gamma:[0,T]\to M$ \emph{has a normal branching} at time $\tau \in (0,T)$ if there exist $\varepsilon_0>0$ and a normal geodesic $\bar{\gamma}:[0,\tau+\varepsilon_0]\to M$ such that $\bar{\gamma} |_{[0,\tau]} = \gamma|_{[0,\tau]}$ and $\bar{\gamma} |_{[0,\tau+\varepsilon]} \neq \gamma|_{[0,\tau+\varepsilon]}$  for all $0<\varepsilon \le \varepsilon_0$.
\end{definition}

To recall the characterization of normal branching, let us introduce the definition of normal affine space and abnormal vector space associated with a horizontal path. 

\begin{definition}
Let $\gamma:[0,T] \to M$ be a horizontal path. For all $t\in (0,T]$, we define the (possibly empty) affine spaces
    \begin{equation}
    \label{eq:Pit}
       \mathcal{N}_t = \left\{\lambda(0) \in T^*_{\gamma(0)}{M} \;\Big|\;  \lambda(\cdot) \text{ normal lift of } \gamma|_{[0,t]} \right\},
    \end{equation}
and the vector spaces
\begin{equation}
\label{eq:defA}
    \mathcal{A}_t = \left\{ \lambda(0) \in T_{\gamma(0)}^*M \;\Big|\;  \lambda(\cdot) \text{  abnormal lift of  } \gamma|_{[0,t]} \right\}\sqcup\left\{0\right\}.
\end{equation}
If $\gamma$ is normal 
then each $\mathcal{N}_t$ is nonempty and its underlying vector space
is $\mathcal{A}_t$. We call $\mathcal{N}_t$ the \emph{normal affine space} and $\mathcal{A}_t$ the \emph{abnormal vector space} of $\gamma|_{[0,t]}$. 
\end{definition}

At each $t\in(0,T]$, $\mathcal{A}_t$ (resp.\ $\mathcal{N}_t$) is an isomorphic image via $(P_{0,t}^u)^*$ of the vector (resp.\ affine) space generated by the abnormal  (resp.\ normal) Lagrange multipliers of $\gamma|_{[0,t]}$. As such, the function $(0,T] \ni t \mapsto \dim(\mathcal{A}_t)$ is upper semi-continuous and non-increasing,  hence it is constant out of a finite set of points, where the function is left-continuous and jumps downwards.
 
Branching of normal geodesics is characterized as follows.  

\begin{theorem}[\cite{Mietton-Rizzi}]
\label{th:n_br}
    A normal geodesic $\gamma:[0,T]\to M$ has a normal branching at time $\tau \in (0,T)$ if and only if the function $t \mapsto \dim(\mathcal{A}_t)$ is discontinuous at $\tau$.
\end{theorem}

We emphasize that, according to this characterization, branching 
of a normal geodesic corresponds to a loss of dimension of the space of abnormal lifts of its subsegments. This implies that a normal geodesic can have normal branching only a finite number of times.

\section{Regular abnormal geodesics}
\label{sect:regular}
Let $(\Delta,g)$ be a sub-Riemannian structure with $\rank(\Delta) =2$ on a smooth manifold $M$. To simplify the exposition, we suppose for now that $(\Delta,g)$ possesses a global orthonormal frame $\{X_1,X_2\} \subset \Vect(M)$, so that the control associated with any horizontal curve is unique. 
We will discuss in which sense the assumption is not restrictive 
in Section~\ref{sect:local_branching}.
At each $q \in M$, for all $m \in \N$ we define
\begin{equation}
    \Delta_m(q)=\spn\left\{ X_{i_1\dots i_m}(q) \mid i_1,\dots,i_m \in \{1,2\}\right\},
\end{equation}
and we denote by $\Delta_m^\perp(q) \subset T^*_qM$ the annihilator of $\Delta_m(q)$ for all $m \in \N$. Notice that, by definition, $\Delta_1(q)=\Delta(q)$ for all $q \in M$. We also set
\begin{equation}
    \Delta_m^\perp= \{ \lambda \in T^*M \mid \lambda \in \Delta_m^\perp(\pi(\lambda)) \} \quad \forall m \in \N.
\end{equation}
Observe that $\Delta^\perp$ is a codimension $2$ submanifold of $T^*M$, but in general $\Delta_m^\perp$ is not a submanifold for $m>1$. \\

Let $\gamma:[0,T] \to M$ be a nowhere constant abnormal path, with control $u \in L^2([0,T],\R^2)$ and let $\lambda(\cdot)$ be an abnormal lift of $\gamma$. Recall that, for all $i_1,\dots, i_m \in \{1,2\}$, it holds
\begin{equation}
    \frac{d}{dt}h_{i_1\dots i_m}(\lambda(t)) = u_1(t) h_{1i_1 \dots i_m}(\lambda(t)) + u_2(t) h_{2i_1\dots i_m}(\lambda(t)) \quad \text{for a.e. } t \in[0,T].
\end{equation}
Since $\gamma$ is abnormal, for all $t \in [0,T]$ it holds $h_1(\lambda(t))=h_2(\lambda(t) )= 0$; differentiating these equalities and using that $u(t) \neq 0$ for a.e.\ $t$, we deduce that $h_{12}(\lambda(t))\equiv 0$ (this is known as the \emph{Goh condition}), which can be rewritten as
\begin{equation}\label{eq:Goh}
    \lambda(t) \in \Delta_2^\perp \quad \forall t \in [0,T].
\end{equation}
Let us now describe the local structure of $\Delta_2^\perp$. This is, in general, only a closed subset of $\Delta^\perp$ without manifold structure, defined by the equations $h_1=h_2=h_{12}=0$. Notice that at $\lambda \in \Delta_2^\perp$ we have $\mathrm{span}\{\vec{h}_1(\lambda),\vec{h}_2(\lambda)\} \subset T_\lambda \Delta^\perp $. This implies that, as soon as $\lambda \in \Delta_2^\perp\setminus \Delta_3^\perp$, the differential $d_\lambda h_{12}$ restricted to $T_\lambda(\Delta^\perp)$ becomes nonzero, since $\scal{dh_{12}}{\vec{h}_i}=\omega(\vec{h}_i,\vec{h}_{12})=h_{i12}$ for $i=1,2$, and at least one between $h_{112}$ and $h_{212}$ does not vanish outside of $\Delta_3^\perp$. In particular, 
\begin{equation}
    \mathcal{R} := \Delta_2^\perp \setminus \Delta_3^\perp 
\end{equation}
is a submanifold of $\Delta^\perp$ of codimension $1$, that is called the \emph{regular abnormal set}.
This submanifold is naturally equipped with the smooth line bundle
\begin{equation}
    \mathrm{span}\{\vec{h}_1(\lambda),\vec{h}_2(\lambda)\}\cap T_\lambda \mathcal{R} \quad \forall \lambda \in \mathcal{R},
\end{equation}
which induces a foliation of $\mathcal{R}$. In particular, any lift of an abnormal path on $M$ that lies in $\mathcal{R}$ must be contained in one of the leaves of this foliation. Let us now define the Hamiltonian $G: \mathcal{R} \to \R$ by
\begin{equation}\label{eq:G}
    G(\lambda) = \frac{-h_{212}(\lambda)}{\sqrt{h_{212}(\lambda)^2+h_{112}(\lambda)^2}} \ h_1(\lambda) + \frac{h_{112}(\lambda)}{\sqrt{h_{212}(\lambda)^2+h_{112}(\lambda)^2}} \ h_2(\lambda),
\end{equation}
which verifies
\begin{equation}\label{eq:vecG}
    \vec{G}(\lambda) = \frac{-h_{212}(\lambda)}{\sqrt{h_{212}(\lambda)^2+h_{112}(\lambda)^2}} \ \vec{h}_1(\lambda) + \frac{h_{112}(\lambda)}{\sqrt{h_{212}(\lambda)^2+h_{112}(\lambda)^2}} \ \vec{h}_2(\lambda) \quad \forall \lambda \in \mathcal{R}.
\end{equation}
Notice that an absolutely continuous path $\lambda:[0,T] \to T^*M$ is an integral line of $\vec{G}$ if and only if $-\lambda(\cdot)$ is an integral line of $-\vec{G}$. The above defined foliation of $\mathcal{R}$ is given by the maximal integral lines of $\vec{G}$. Their projections on $M$ are smooth abnormal paths parametrized with unit speed. We refer to \cite{liu-sus} for further details on this construction. Notice that the flow of $\vec{G}$ is well-defined as long as the integral lines do not intersect $\Delta_3^\perp$. For a description of how an integral line of $\vec{G}$ can approach $\Delta_3^\perp$ in finite time, see \cite{sss5}.
\begin{definition}
\label{def:rt}
    For all $T\ge 0$ we set 
    \begin{equation}
        \mathcal{R}_T := \{ \lambda \in \mathcal{R} \mid  e^{t\vec{G}}(\lambda) \in \mathcal{R} \ \text{for all } t \in [0,T] \}.
    \end{equation}
    We also define 
    \begin{equation}
        \mathcal{D} := \bigsqcup_{T \ge 0} \{T\} \times \mathcal{R}_T.
            \end{equation}
\end{definition}
 By continuous dependence of the flow of $\vec{G}$ on initial data, $\mathcal{R}_T$ is an open subset of $\mathcal{R}$ for all $T\ge 0$, with $\mathcal{R}_0=\mathcal{R}$, and $\mathcal{D}$ is relatively open in $[0,+\infty) \times \mathcal{R}$. By definition, $\mathcal{D}$ is the largest subset of $[0,+\infty) \times \mathcal{R}$ where the map $(t,\lambda)~\mapsto~e^{t \vec{G}}(\lambda) \in \mathcal{R}$ is well-defined.

\begin{definition}
\label{def:reg_abn}
    An abnormal path $\gamma:[0,T] \to M$ is called \emph{regular} if it is the projection on $M$ of an integral line of $\vec{G}$ or $-\vec{G}$. Such lift is called a \emph{positively oriented} (resp.\ \emph{negatively oriented}) regular abnormal lift if it is an integral line of $\vec{G}$ (resp.\ $-\vec{G}$).
\end{definition}

Notice that if $\gamma$ is a regular abnormal path then $\Delta_2(\gamma(t)) \neq \Delta_3(\gamma(t))$ for all $t$.
For \emph{geodesics}, abnormal lifts must be regular as long as they remain out of $\Delta_3^\perp$, and are thus uniquely determined by their initial covector, as described in the next proposition.

\begin{proposition}\label{prop:Imax}
    Let $\gamma:[0,T]\to M$ be an abnormal geodesic parametrized with unit speed. Let $\lambda :[0,T]\to T^*M$ be an abnormal lift of $\gamma$.
    Assume that $\lambda(t_0)\in \mathcal{R}$ for some $t_0\in [0,T]$, and let $I\subset [0,T]$ be an interval containing $t_0$ such that $\lambda(t) \in \mathcal{R}$ for all $t\in I$. Then, $\lambda|_{I}$ is the unique 
abnormal lift of $\gamma|_{I}$ that passes through $\lambda(t_0)$. Furthermore, it is 
regular abnormal.
    
    \end{proposition}
\begin{proof}
Let $u \in L^2([0,T],\mathbb{S}^1)$ be the control of $\gamma$. As we have shown in \eqref{eq:Goh}, if $\lambda(\cdot)$ is an abnormal lift, then $h_{12}(\lambda(t))=0$ for all $t\in[0,T]$. Differentiating this identity we find
\begin{equation}
    u_1(t)h_{112}(\lambda(t))+ u_2(t)h_{212}(\lambda(t))=0 \quad \text{ for a.e. } t \in [0,T].
\end{equation}
Let
\begin{equation}
    \h(\lambda):=(-h_{212}(\lambda),h_{112}(\lambda)) \quad \forall \lambda \in T^*M,
\end{equation} 
and note that $\lambda(t)\in \mathcal{R}$ if and only if $\lambda(t) \notin \Delta_3^\perp$, which, in turn, is equivalent to $\h(\lambda(t))\neq0$. 
At all such $t$, provided that $\gamma$ is differentiable, it holds
\begin{equation}
\label{eq:contr_ra}
    u(t)=\pm \frac{\h(\lambda(t))}{|\h(\lambda(t))|},
\end{equation}
where $\lvert \h(\lambda)\rvert = \sqrt{h_{212}(\lambda)^2 + h_{112}(\lambda)^2}$. In particular there exists a measurable function $\sigma:I \to \{-1,1\}$ such that 
\begin{equation}
    \dot{\lambda}(t) = \sigma(t) \vec{G}(\lambda(t)) \quad \text{for a.e. } t \in I.
\end{equation}
By the geodesic assumption, short arcs of $\gamma$ are length minimizers, which implies that $\sigma$ is constant (see, e.g., the argument in \cite[Lemma 22]{sss5}).
The fact that $\sigma$ is constant means that $\lambda|_{I}$ is an integral line of $\vec{G}$ or $-\vec{G}$.
\end{proof}

 \begin{remark}[Regular vs nice]
Abnormal paths having an abnormal lift lying in $\mathcal{R}$ are called \emph{nice} in \cite{ABB}. 
Accordingly, all regular abnormal paths are nice, but the converse does not hold. The above proposition implies that, for unit-speed geodesics, being nice is equivalent to being regular.
\end{remark}

The following result motivates the interest in the class of regular abnormal paths. 

\begin{theorem}[\cite{liu-sus,ABB}]
\label{pr:liu-sus}
    Every regular abnormal path is a geodesic. 
\end{theorem}

We conclude this section by stating a geometric property of regular abnormal lifts. 

\begin{proposition}
\label{prop:cones}
    Let $\gamma:[0,T] \to M$ be a regular abnormal path. The set of positively (resp.\ negatively) oriented regular abnormal lifts of $\gamma$ is a convex cone.
\end{proposition}

\begin{proof} It is sufficient to prove the statement for positively oriented lifts. We divide the proof into two claims. \\

    \emph{Claim 1}: if $\lambda(\cdot)$ is a positively oriented regular abnormal lift and $\alpha>0$ then also $\alpha \lambda(\cdot)$ is a regular abnormal lift with the same orientation. \\
    
    \emph{Proof of Claim 1.} For $\alpha \in \R\setminus\{0\}$ let  $P_\alpha : T^*M \to T^*M$ be the fiber-wise rescaling $P_\alpha(\mu) := \alpha \mu$. It holds $h_{i_1\dots i_j} \circ P_\alpha = \alpha h_{i_1 \dots i_j}$, and thus $(P_\alpha)_*\vec{h}_i = \vec{h}_i$, where the star denotes the push-forward. Omitting evaluation at $t$, it follows that
    \begin{align}
        \frac{d}{dt}(P_{\alpha}\circ\lambda) & = -\frac{h_{212}(\lambda)}{|\h(\lambda)|} \ \vec{h}_1(\alpha \lambda) + \frac{h_{112}(\lambda)}{|\h(\lambda)|} \ \vec{h}_2(\alpha \lambda)\\
        & = \mathrm{sign}(\alpha)\left(-\frac{h_{212}(\alpha\lambda)}{|\h(\alpha\lambda)|} \ \vec{h}_1(\alpha \lambda) + \frac{h_{112}(\alpha\lambda)}{|\h(\alpha\lambda)|} \ \vec{h}_2(\alpha \lambda)\right) = \mathrm{sign}(\alpha) \vec{G}(\alpha \lambda).
    \end{align}
    The claim follows. \\

    \emph{Claim 2:} if $\lambda_1(\cdot),\lambda_2(\cdot)$ are positively oriented regular abnormal lifts, then for all $\alpha_1,\alpha_2 >0$, the curve $\alpha_1 \lambda_1(\cdot) + \alpha_2 \lambda_2(\cdot)$ is a positively oriented regular abnormal lift. \\

    \emph{Proof of Claim 2.} 
    Let $u$ be the control associated with $\gamma$, which is smooth by definition of regular abnormal. Notice that $\h(\lambda_i(t))= |\h(\lambda_i(t))|u(t)$ for all $t \in [0,T]$, $i=1,2$. Therefore it holds
    \begin{equation}
        \h(\alpha_1\lambda_1(t)+\alpha_2\lambda_2(t))=(\alpha_1|\h(\lambda_1(t))|+\alpha_2|\h(\lambda_2(t))|)u(t) \neq 0 
    \end{equation}
    for all $t \in [0,T]$. Exploiting the fact that
    \begin{equation}
        \frac{\h(\lambda_1(t))}{|\h(\lambda_1(t))|}= \frac{\h(\lambda_2(t))}{|\h(\lambda_2(t))|}=u(t) \quad \forall t \in [0,T]
    \end{equation}
    we can deduce that 
    \begin{equation}
        |\h(\alpha_1 \lambda_1(t) + \alpha_2 \lambda_2(t))|=\alpha_1|\h(\lambda_1(t))|+\alpha_2|\h(\lambda_2(t))|,
    \end{equation}
    which means that \eqref{eq:contr_ra} holds along $\alpha_1 \lambda_1(\cdot) + \alpha_2 \lambda_2(\cdot)$ with constant positive sign, hence $\alpha_1 \lambda_1(\cdot) + \alpha_2 \lambda_2(\cdot)$ is an integral line of $\vec{G}$.
\end{proof}

\subsection{Branching of regular abnormal geodesics} 
Let $\gamma:[0,T]\to M$ be a regular abnormal geodesic. As we will show, the fact that $\gamma$ has a regular abnormal branching is characterized by a loss of dimension of the vector spaces $\mathcal{A}_t$, defined in \eqref{eq:defA}. Let us begin with a description of these sets. 
\begin{definition}
For any $t \in (0,T]$, we decompose $\mathcal{A}_t$ as the disjoint union
\begin{equation}
\label{eq:decomposition}
    \mathcal{A}_t= \mathrm{REG}(t) \sqcup \tnr(t) \sqcup \mathrm{L}(t)\sqcup\{0\},
\end{equation}
where:
\begin{itemize}
    \item $\mathrm{REG}(t)= \mathrm{REG}^+(t) \sqcup \mathrm{REG}^-(t)$ is the set of initial values of regular abnormal lifts of $\gamma|_{[0,t]}$, and $\mathrm{REG}^+(t)$ (resp. $\mathrm{REG}^-(t)$) are the initial values of the positively (resp. negatively) oriented ones;
    \item $\tnr(t)$ is the set of initial values of abnormal lifts of $\gamma|_{[0,t]}$ in $\Delta_3^\perp$. The notation stands for \emph{totally non-regular};
    \item $\mathrm{L}(t)$ is the set of initial values of all the remaining abnormal lifts of $\gamma|_{[0,t]}$.
\end{itemize}
\end{definition}

An example of the decomposition \eqref{eq:decomposition} can be found in Section \ref{sec:exampledecomposition} below.

\begin{remark}
\label{rk:+-G}  Positive or negative orientation of a regular abnormal lift is determined by its initial value: the positively oriented ones are the orbits $s \mapsto e^{s\vec{G}}(\lambda_0)$ with $\lambda_0 \in \mathrm{REG}^+(t)$, while the negative oriented ones are the orbits $s \mapsto e^{-s\vec{G}}(\lambda_0)$ with $\lambda_0 \in \mathrm{REG}^-(t)$. Multiplication by $-1$ yields a one-to-one correspondence between $\mathrm{REG}^+(t)$  and $\mathrm{REG}^-(t)$.
\end{remark}

Let us collect some properties of $\mathrm{REG}(t)$ and $\tnr(t)$. 
\begin{proposition}
\label{prop:decomp}
    For all $0<t \leq T$ and $0 < t_1 \le t_2 \le T$, the following hold:
    \begin{itemize}
        \item [(i)] $\mathrm{REG}(t_2) \subset \mathrm{REG}(t_1)$ and $\tnr(t_2) \subset \tnr(t_1)$;
        \item [(ii)] $\mathrm{TNR}(t) \sqcup \{0\}$ is a vector subspace of $\mathcal{A}_t$;  in particular, the function 
        \begin{equation}
            t \mapsto \dim(\tnr(t) \sqcup \{0\})
        \end{equation}is piecewise constant, non-increasing and left-continuous;
        \item [(iii)] $\mathrm{REG}^+(t)$ and $\mathrm{REG}^-(t)$ are open convex cones inside $\mathcal{A}_t$;
        \item [(iv)] $\mathrm{REG}^{\pm}(t)+\tnr(t)=\mathrm{REG}^{\pm}(t)$, i.e. for all $\lambda_0 \in \mathrm{REG}^{\pm}(t)$ and $\xi_0 \in \tnr(t)$ it holds $\lambda_0+\xi_0 \in \mathrm{REG}^{\pm}(t)$.
    \end{itemize}
\end{proposition}

Before proving this result, let us state a technical lemma. 

\begin{lemma}
\label{lem:decomp}
   Let $\gamma:[0,T]\to M$ be a regular abnormal path, and fix a regular abnormal lift $\lambda:[0,T] \to T^*M$ of $\gamma$. Set $\lambda_0 :=\lambda(0)$. For all $t \in (0,T]$ it holds
    \begin{equation}
        \mathcal{A}_t= \spn\{\lambda_0\} \oplus (\Delta_3^\perp\cap \mathcal{A}_t). 
    \end{equation}
\end{lemma}
\begin{proof}
    The inclusion "$\supset$" is trival. Let us prove that also "$\subset$" holds true. Let $\mu(\cdot)$ be an abnormal lift of $\gamma|_{[0,t]}$, with $\mu_0:=\mu(0) \notin \Delta_3^\perp$. By Theorem~\ref{pr:liu-sus} and Proposition~\ref{prop:Imax} there exists $\varepsilon >0$ such that $\mu|_{[0,\varepsilon]}$ is a regular abnormal lift of $\gamma|_{[0,\varepsilon]}$. Let $u$ be the control associated with $\gamma$. Recalling the characterization of the controls of regular abnormal geodesics \eqref{eq:contr_ra},  without loss of generality we assume $\h(\mu_0)=| \h(\mu_0)|u(0)$ (i.e. $\mu|_{[0,\varepsilon]}$ is positively oriented by Remark \ref{rk:+-G}). Up to replacing the reference regular abnormal lift $\lambda(\cdot)$ with $-\lambda(\cdot)$,  we can assume that the same relation holds for $\lambda_0$. Then, the following holds:
    \begin{align}
        \h\left( \lambda_0- \frac{|\h(\lambda_0)|}{|\h(\mu_0)|}\mu_0\right ) 
        & = \h(\lambda_0) - \frac{|\h(\lambda_0)|}{|\h(\mu_0)|}\h(\mu_0) \\
        & = |\h(\lambda_0)|u(0) - \frac{|\h(\lambda_0)|}{|\h(\mu_0)|}| \h(\mu_0)|u(0) =0.
    \end{align}
    Thus $\xi(\cdot):=\lambda(\cdot) -\tfrac{|\h(\lambda_0)|}{|\h(\mu_0)|} \mu(\cdot)$, is an abnormal lift of $\gamma$ with $\xi(0)\in \Delta_3^\perp\cap \mathcal{A}_t$.
\end{proof}

Lemma \ref{lem:decomp} implies that if we fix a reference regular abnormal lift $\lambda(\cdot)$ of $\gamma$, any regular abnormal lift  $\mu(\cdot)$ of a subsegment $\gamma|_{[0,t]}$ is characterized, up to a factor $\alpha \neq0$, by a unique $\xi_0 \in \Delta_3^\perp \cap \mathcal{A}_t$, such that, letting $\lambda_0=\lambda(0)$, $\mu_0=\mu(0)$, it holds  $\mu_0=\alpha \lambda_0+\xi_0$. However, for fixed $\alpha \neq 0$ it is not true, for a general $\xi_0 \in \Delta_3^\perp \cap \mathcal{A}_t$, that $\alpha \lambda_0+\xi_0$ is the initial value of a regular abnormal lift of $\gamma|_{[0,t]}$, unless $\xi_0$ is \emph{small enough}, as we prove below.

\begin{proof}[Proof of Proposition \ref{prop:decomp}]
    Properties $(i),(ii)$ are straightforward.
    
    Concerning $(iii)$, it was proved in Proposition~\ref{prop:cones} that $\mathrm{REG}^{\pm}(t)$ are convex cones. Let us prove that $ \mathrm{REG}^+(t)$ is open.  Let $\lambda_0 \in \mathrm{REG}^+(t)$ be the initial value of a regular abnormal lift $\lambda(\cdot)$ of $\gamma|_{[0,t]}$. There exists $C>0$ such that $|{\h}(\lambda(s))| \ge C$ for all $s \in [0,t]$.  Thanks to Lemma~\ref{lem:decomp}, we can write any $\mu_0 \in \mathcal{A}_t$ as  $\mu_0=\alpha \lambda_0+\xi_0$ with $\xi_0 \in \Delta_3^\perp \cap \mathcal{A}_t$ and $\alpha \in \R$. Working in a coordinate chart centered at $\gamma(0)$, we can choose  $\rho>0$ such that, defining
    \begin{equation} 
        \mathcal{B} = \left\{ \xi_0 \in \Delta_3^\perp \cap \mathcal{A}_t \;\Big|\; |\xi_0| < \rho \right\},
    \end{equation}
    the quantity
    \begin{equation}
        M := \sup \left\{ |\h(\xi(s))| \;\big|\; s \in [0,t], \ \xi(\cdot) \ \text{abnormal lift of $\gamma|_{[0,t]}$ s.t. } \xi(0)=\xi_0 \in \mathcal{B} \right\}
    \end{equation}
    verifies $M < C$. Here $|\xi_0|$ denotes the Euclidean norm of $\xi_0$, identified with a vector in $\R^n$ via the coordinate chart. Now choose $\varepsilon>0$ so that $(1-\varepsilon)C>M$ and let  
    \begin{equation}
        \Omega = \{\alpha \lambda_0 + \xi_0 \mid \alpha > 1-\varepsilon, \ \xi_0 \in \mathcal{B} \} \subset \mathcal{A}_t.
    \end{equation}
    By Lemma~\ref{lem:decomp}, $\Omega$ is an open neighborhood of $\lambda_0$ in $\mathcal{A}_t$.  Let $\mu_0 = \alpha \lambda_0 + \xi_0\in \Omega$ with $\alpha, \xi_0$ as in the definition of $\Omega$. Denoting by $\xi(\cdot)$ the unique abnormal lift of $\gamma|_{[0,t]}$ such that $\xi(0)=\xi_0$, the abnormal lift of $\gamma|_{[0,t]}$ associated with $\mu_0$ is $\mu(\cdot):=\alpha\lambda(\cdot) + \xi(\cdot)$, which verifies
    \begin{align}
        |\h(\mu(s))|  \ge \alpha |\h(\lambda(s))| - |\h(\xi(s))|   
        \ge (1-\varepsilon) C - M > 0
    \end{align}
    for all $s \in [0,t]$, i.e. $\mu_0 \in \mathrm{REG}(t)$. Moreover, $\mu_0 \in  \mathrm{REG}^+(t)$ thanks to Remark~\ref{rk:+-G} and to the fact that $\alpha >0$.
    
    Finally, the fact that $\mathrm{REG}^{\pm}(t)+\tnr(t) \subset \mathrm{REG}^{\pm}(t)$ follows by observing that if $\lambda(\cdot)$ is a regular abnormal lift of $\gamma$ and $\xi(\cdot)$ is an abnormal lift of $\gamma$ in $\Delta_3^\perp$ then their sum is an abnormal lift of $\gamma$. Since it never intersects $\Delta_3^\perp$, it is regular by Proposition \ref{prop:Imax}. Moreover, it has the same orientation as $\lambda(\cdot)$.
\end{proof}

\begin{definition}
\label{def:ra_br}
    We say that a regular abnormal geodesic $\gamma:[0,T]\to M$ \emph{has a regular abnormal branching} at time $\tau \in (0,T)$ if there exist $\varepsilon_0>0$ and a regular abnormal geodesic $\bar{\gamma}:[0,\tau+\varepsilon_0]\to M$ such that $\bar{\gamma} |_{[0,\tau]} = \gamma|_{[0,\tau]}$ and $\bar{\gamma} |_{[0,\tau+\varepsilon]} \neq \gamma|_{[0,\tau+\varepsilon]}$  for all $0<\varepsilon\le\varepsilon_0$.
\end{definition}

The following observation is a consequence of the description of regular abnormal geodesics given in the previous section.

\begin{proposition}
\label{prop:analytic}
    In real-analytic sub-Riemannian manifolds of rank two, regular abnormal geodesics cannot have either normal or regular abnormal branching.
\end{proposition}

\begin{proof}
   Under the assumptions, normal and regular abnormal geodesics are real-analytic curves, since they are projections of integral lines of $\vec{H}$ and $\vec{G}$, respectively. Therefore, they cannot have normal nor regular abnormal branching. 
\end{proof}

We now prove the main theorem of this section.

\begin{theorem}
\label{th:abn_br}
    A regular abnormal geodesic $\gamma:[0,T] \to M$ has a regular abnormal branching at time $\tau\in (0,T)$ if and only if $t \mapsto \dim(\mathcal{A}_t)$ is discontinuous at $\tau$.
\end{theorem}
\begin{proof}
    Assume that $\gamma$ has a regular abnormal branching at time $\tau$, let $\varepsilon_0>0$ and $\bar{\gamma}:[0,\tau+\varepsilon_0]\to M$ be a regular abnormal geodesic that branches from $\gamma$ at time $\tau$. Let $\bar{\lambda}(\cdot)$ be a regular abnormal lift of $\bar{\gamma}$ with $\bar{\lambda}(0)=:\bar{\lambda}_0$. In particular, $\bar{\lambda}|_{[0,\tau]}$ is an abnormal lift of $\gamma|_{[0,\tau]}$, thus $\bar{\lambda}_0 \in \mathcal{A}_{\tau}$. 
     Let us show that $\bar{\lambda}_0 \not\in \mathcal{A}_{\tau+\varepsilon}$ for $\varepsilon>0$. Indeed, if by contradiction there existed an abnormal lift $\lambda(\cdot)$ of $\gamma|_{[0,\tau+\varepsilon]}$ starting from $\bar{\lambda}_0$,  it should, in addition, coincide with $\bar{\lambda}(\cdot)$ on $[0,\tau]$ (Proposition~\ref{prop:Imax}). 
    As a consequence, up to reducing $\varepsilon$, $\lambda|_{[0,\tau+\varepsilon]}$  should be a regular abnormal lift of $\gamma|_{[0,\tau+\varepsilon]}$ with the same orientation as $\bar\lambda$. This would imply that $\lambda|_{[0,\tau+\varepsilon]}$ and $\bar{\lambda}|_{[0,\tau+\varepsilon]}$ are two distinct integral lines of $\vec{G}$ (or $-\vec{G}$) that coincide on $[0,\tau]$, leading to a contradiction.
   %
   By construction $\mathcal{A}_t\supset \mathcal{A}_{t+\varepsilon}$, and the above argument shows that the inclusion must be strict, proving that $t \mapsto \dim(\mathcal{A}_t)$ is discontinuous at $\tau$.
    
    Conversely, assume that $t \mapsto \dim(\mathcal{A}_t)$ is discontinuous at $\tau$. Since by Proposition \ref{prop:decomp} $\mathrm{REG}(t)$ is open in $\mathcal{A}_t$, it contains a basis of $\mathcal{A}_t$ for all $t \in (0,T]$. Furthermore, $t \mapsto \dim(\mathcal{A}_t)$ is  constant when $t>\tau$ and $t-\tau$ is small. Hence, not all the elements of the basis of $\mathcal{A}_\tau$ contained in $\mathrm{REG}(\tau)$  also belong to $\mathcal{A}_t$ for $t>\tau$. As a consequence, 
    we can find $\lambda_0 \in \mathrm{REG}(\tau)$ such that $\lambda_0 \notin \mathcal{A}_{\tau+\varepsilon}$ for all $\varepsilon>0$. The fact that $\lambda_0 \in \mathrm{REG}(\tau)$ means that, replacing $\lambda_0$ with $-\lambda_0$ if necessary, the curve $[0,\tau] \ni t \mapsto e^{t\vec{G}}(\lambda_0)$ is a regular abnormal lift of $\gamma|_{[0,\tau]}$ and in particular $e^{\tau \vec{G}}(\lambda_0) \in \mathcal{R}$. Thus, there exists $\varepsilon_0>0$ such that the curve $\bar{\gamma}(t):= \pi \circ e^{t\vec{G}}(\lambda_0)$ is well-defined on the interval $[0,\tau+\varepsilon_0]$, and it is a regular abnormal geodesic. Since by construction $\lambda_0\notin \mathcal{A}_{\tau+\varepsilon}$ for all $\varepsilon>0$ we must have $\bar{\gamma}|_{[0,\tau+\varepsilon]}\neq \gamma|_{[0,\tau+\varepsilon]}$ for all $0<\varepsilon\le\varepsilon_0$.
    \end{proof}

\subsection{Explicit description of the branching}
\label{sect:br_descr}
 Let us give a more detailed description of the regular abnormal branching. Let $\gamma:[0,T] \to M$ be a regular abnormal geodesic that has regular abnormal branching at time $\tau \in (0,T)$. Fix a regular abnormal lift $\lambda(\cdot)$ of $\gamma$ and let $\lambda_0:=\lambda(0)$. We can assume that $\lambda(\cdot)$ is positively oriented and, by symmetry, restrict our attention to positively oriented lifts. Fix $\varepsilon>0$ such that $\mathcal{A}_t \equiv \mathcal{A}_{\tau+\varepsilon}$ for all $t \in(\tau,\tau+\varepsilon]$. The proof of Theorem~\ref{th:abn_br} implies that the initial values of the positively oriented regular abnormal lifts of the geodesics that branch from $\gamma$ at time $\tau$ are the covectors $\mu_0 \in \mathcal{R} \cap T^*_{\gamma(0)}M$ such that $\mu_0 \in \mathrm{REG}^+(\tau)$ and $\mu_0 \notin \mathcal{A}_{\tau+\varepsilon}$.
 Recalling Lemma~\ref{lem:decomp}, for all such $\mu_0$ there exist $\beta>0$ and a unique $\xi_0 \in \Delta_3^\perp \cap \mathcal{A}_\tau$ such that 
\begin{equation}
    \beta\mu_0= \lambda_0 +\xi_0.
\end{equation}
Thanks to Proposition \ref{prop:cones}, the family of ($\gamma$ and) the regular abnormal geodesics that branch from $\gamma$ at time $\tau$ is characterized by the open set
\begin{equation}
\label{eq:branching_curves}
    \mathrm{B} := \{ \xi_0 \in \Delta_3^\perp \cap \mathcal{A}_{\tau} \mid \lambda_0 + \xi_0 \in \mathrm{REG}^+(\tau) \setminus \mathcal{A}_{\tau+\varepsilon}\} \cup \{0\}.
\end{equation}
Indeed,  recalling the definition of $\mathcal{R}_T$ and $\mathcal{D}$ (Definition \ref{def:rt}), and letting 
\begin{equation}
    \mathcal{D}(\lambda_0) := \{(t,\xi_0) \in [0,\infty)\times \mathcal{R} \mid (t, \lambda_0+\xi_0) \in \mathcal{D \}},
\end{equation}
we can define the function
\begin{equation}
    \Phi: ([0,T]\times \mathrm{B}) \cap \mathcal{D}(\lambda_0) \to M, \quad \Phi(t,\xi_0):= \pi \circ e^{t \vec{G}}(\lambda_0 + \xi_0).
\end{equation}
By construction, $\Phi(\cdot,0)=\gamma(\cdot)$, while the maps $\Phi(\cdot,\xi_0)$ with $\xi_0 \neq 0$ are all the regular abnormal geodesics that branch from $\gamma$ at time $\tau$, restricted to the interval $[0,T]$.

\begin{remark} The following observation will be useful in a coming example. In the above setting, and assuming that $\mathcal{A}_t=\spn\{\lambda_0\} \oplus \mathrm{TNR}(t)$ for all $t \in (0,T]$, then $\Delta_3^\perp \cap \mathcal{A}_t = \mathrm{TNR}(t)$ for all $t \in (0,T]$ and 
\begin{equation}
    \mathrm{B}= \left(\tnr(\tau) \setminus \tnr(\tau+\varepsilon)\right) \cup \{0\}.
\end{equation} 
\end{remark}

\begin{remark}
\label{rk:Sard}
    The set $\mathrm{B}$ defined in \eqref{eq:branching_curves} is an open subset of $\Delta_3^\perp \cap \mathcal{A}_{\tau}$, whose dimension is at most $n-3$ (note that $\Delta_3(\gamma(0)) \neq \Delta_2(\gamma(0))$ since $\gamma$ is a regular abnormal path). As a consequence, the image of the smooth map $\Phi$ has Hausdorff dimension at most $n-2$. 
\end{remark}

\begin{subsection}{An example of the decomposition}\label{sec:exampledecomposition}
We illustrate the decomposition of $\mathcal{A}_t$ defined in \eqref{eq:decomposition}. On $\R^5$ we consider the rank-two sub-Riemannian structure generated by the orthonormal frame
\begin{equation}
    X_1=\partial_1, \quad X_2= \partial_2 + \frac{x_1^2}{2}\partial_3 + \frac{x_1^3}{6}\partial_4 + x_2x_3\partial_5.
\end{equation}
Fix $T>0$ and consider the horizontal curve
\begin{equation}
    \gamma:[0,T] \to \R^5, \quad \gamma(t) = (0,t,0,0,0).
\end{equation}
The abnormal lifts of $\gamma$ are solutions of $\dot{\lambda}(t)=\vec{h}_2(\lambda(t))$, which we write in coordinates $\lambda(t)=(x(t),p(t))$ with $\lambda(0)=((0,0,0,0,0),(\bar{p}_1,\bar{p}_2,\bar{p}_3,\bar{p}_4,\bar{p}_5))$. Writing $\vec{h}_2$ explicitly, we find that $p(\cdot)$ satisfies
\begin{align}
    \begin{cases}
        \dot{p}_1(t)=\dot{p}_2(t)=0, \\
        \dot{p}_3(t)=-tp_5(t), \\
        \dot{p}_4(t)=\dot{p}_5(t)=0,
    \end{cases}
\end{align}
hence  $p_j(t) \equiv \bar{p}_j$ for $j=1,2,4,5$ and $p_3(t)=\bar{p}_3 - \frac{t^2}{2}\bar{p}_5$. Since
\begin{equation}
    X_2(\gamma(t))  =\partial_2,\quad 
    X_{12}(\gamma(t))  =0, \quad     X_{112}(\gamma(t)) = \partial_3, \quad 
    X_{212}(\gamma(t)) =0,
\end{equation}
abnormal lifts must satisfy $\bar{p}_1=\bar{p}_2=0$. We conclude that the abnormal lifts of $\gamma$ are the $3$-parameter family
\begin{equation}
    (x(t),p(t))= \left((0,t,0,0,0),(0,0,\bar{p}_3 - \frac{t^2}{2} \bar{p}_5, \bar{p}_4,\bar{p}_5)\right), \quad (\bar{p}_3,\bar{p}_4,\bar{p}_5) \in \R^3 \setminus \{0\}.
\end{equation}
For all $t \in (0,T]$ we have
\begin{equation}
    \mathcal{A}_t  \cong \R^3_{(\bar{p}_3,\bar{p}_4,\bar{p}_5)},
\end{equation}
while $\mathrm{REG}(t)$ is  characterized by the condition $\bar{p}_3 - \frac{s^2}{2} \bar{p}_5 \neq 0$ for all $s \in [0,t]$, and it is then the union of the two  open convex cones
\begin{equation}
    \mathrm{REG}^+(t) = \{ \bar{p}_3>0, \ \bar{p}_5<2\bar{p}_3/t^2 \}, \qquad \mathrm{REG}^-(t) =  \{ \bar{p}_3<0, \ \bar{p}_5>2\bar{p}_3/t^2 \}
\end{equation}
that are decreasing in time (with respect to inclusion). Notice that $\Delta_3^\perp \cap \mathcal{A}_t$ is the two-dimensional subspace defined by the equation $\bar{p}_3=0$, while the set 
\begin{equation}
    \tnr(t)=\{\bar{p}_3=\bar{p}_5=0\} \setminus \{0\}
\end{equation}
of 
initial data of totally non-regular lifts is one-dimensional and independent of $t$.
Finally, 
\begin{equation}
    \mathrm{L}(t) = \left ( \{ \bar{p}_3 \ge 0, \ \bar{p}_5 \ge 2\bar{p}_3/t^2 \} \cup \{ \bar{p}_3\le0, \ \bar{p}_5\le 2\bar{p}_3/t^2\} \right ) \setminus \{\bar{p}_3=\bar{p}_5=0 \} .
\end{equation}
We note that $t\mapsto \dim(\mathcal{A}_t)$ is constant so that there is no regular abnormal branching in this example. For such examples, we refer to Section \ref{sec:examplesofbranching}.
\end{subsection}

\begin{figure}[H]
     \centering
     \includegraphics[width=220pt]{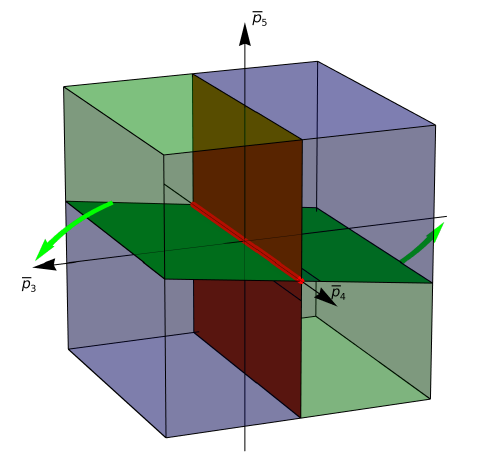}
     \caption{Plot of $\mathrm{REG}(t)$ (blue cone), $\mathrm{TNR}(t)$ (red line), $\mathrm{L}(t)$ (green cone), and $\Delta_3^\perp\cap\mathcal{A}_t$ (vertical plane). The green arrows indicate how the interface between $\mathrm{REG}(t)$ and $\mathrm{L}(t)$ evolves as $t$ increases.}
\end{figure}

\subsection{Branching and strict abnormality}

Let $\gamma:[0,T] \to M$ be a regular abnormal geodesic, and assume that it has regular abnormal branching at time $\tau \in (0,T)$. We provide additional conditions under which the regular abnormal geodesics branching from $\gamma$ are also strictly abnormal. First, we deal with the case when the \emph{common} subsegment $\gamma|_{[0,\tau]}$ itself is strictly abnormal. 

\begin{proposition}
\label{prop:ex1}
    Let $\gamma:[0,T] \to M$ be a regular abnormal geodesic  that branches at time $\tau \in (0,T)$. If $\gamma|_{[0,\tau]}$ is strictly abnormal, then $\gamma$ and all the curves branching from it are strictly abnormal geodesics.  
\end{proposition}

\begin{proof}
   If a subsegment of an abnormal path has no normal lifts, then also the whole path has no normal lifts.
\end{proof}

Let us now deal with the case when $\gamma$ is both normal and regular abnormal. Combining Theorems \ref{th:n_br} and \ref{th:abn_br} we deduce that $\gamma$ has a normal branching at time $t$ if and only if it has a regular abnormal branching at the same time. In general, some of the geodesics that branch from $\gamma$ might be both normal and regular abnormal (while others are either strictly normal or strictly abnormal). However, in the special case that we now describe, the jump of dimension of $\mathcal{A}_t$ gives rise to distinct family of curves. In this case, the regular abnormal geodesics that branch from $\gamma$ are strictly abnormal. 

\begin{proposition}
\label{prop:strictn}
    Let $\gamma:[0,T] \to M$ be a normal and  regular abnormal geodesic. Assume that the function $t \mapsto \dim(\mathcal{A}_t)$ is discontinuous at time $\tau \in (0,T)$, with   
    \begin{equation}
    \label{eq:ra&n}
        \dim(\mathcal{A}_{\tau})= k+1, \quad \dim(\mathcal{A}_t)=k \quad \forall t \in (\tau,T],
    \end{equation}
    for some $k \ge 1$.
    Let $\bar{\gamma}:[0,T] \to M$ be a normal geodesic that branches from $\gamma$ at time $\tau$. Then, $\bar{\gamma}$ does not admit a regular abnormal lift. In particular, the regular abnormal geodesics that branch from $\gamma$ at time $\tau$ are strictly abnormal.
\end{proposition}

\begin{proof}
    By contradiction, assume that for some $\varepsilon >0$ the curve $\bar{\gamma}|_{[0,\tau+\varepsilon]}$ is also regular abnormal. Let $\mathcal{A}_t, \bar{\mathcal{A}}_t$ be the abnormal vector spaces and $\mathcal{N}_t, \bar{\mathcal{N}}_t$ be the normal affine spaces associated with $\gamma$ and $\bar{\gamma}$, and notice that $\mathcal{N}_t=\bar{\mathcal{N}}_t$ for $t \in [0,\tau]$. By hypothesis, $\mathcal{N}_{\tau}$ has dimension $k+1$, while, up to reducing $\varepsilon$, for $t\in (\tau,\tau+\varepsilon]$, $\mathcal{N}_t$ and $\bar{\mathcal{N}}_t$ are respectively a $k$-dimensional and $m$-dimensional (with $1\le m\le k$) constant affine subspaces of $\mathcal{N}_t$, and either they intersect at some point or they are parallel. In particular, if $\mathcal{N}_t$ and $\bar{\mathcal{N}}_t$ are parallel then $\mathcal{A}_t$ must contain $\bar{\mathcal{A}}_t$. Observe that,  for $t \in (\tau,\tau+\varepsilon]$:
    \begin{itemize}
        \item if there exists $\lambda_0 \in \mathcal{N}_t \cap \bar{\mathcal{N}}_t$ then the curve $[0,t] \ni s \mapsto e^{s \vec{H}}(\lambda_0)$ is a normal lift of both $\gamma|_{[0,t]}$ and $\bar{\gamma}|_{[0,t]}$, which is impossible since their supports branch at $\gamma(\tau)$; 
        \item if instead $\mathcal{N}_t$ and $\bar{\mathcal{N}}_t$ are parallel, there exists $\lambda_0 \in \mathcal{A}_t \cap \bar{\mathcal{A}}_t \cap \mathcal{R}$, since $\bar{\gamma}|_{[0,\tau+\varepsilon]}$ is regular abnormal under our working assumption and $\bar{\mathcal{A}}_t \subset \mathcal{A}_t$. Following Proposition \ref{prop:Imax} we infer that the curve  $[0,t]\ni s\mapsto e^{s\vec{G}}(\lambda_0)$ (or $s \mapsto e^{-s\vec{G}}(\lambda_0)$) is a regular abnormal lift of both $\gamma|_{[0,t]}$ and $\bar{\gamma}|_{[0,t]}$, which also leads to a contradiction.
    \end{itemize} 
    Therefore, $\bar{\gamma}$ admits no regular abnormal lift. 
\end{proof}


\section{Examples of regular strictly abnormal branching}
\label{sec:examplesofbranching}
In this section we give two examples of branching of regular abnormal geodesics. In the two cases,  as an application of the criteria obtained either in Proposition~\ref{prop:ex1} or in Proposition~\ref{prop:strictn}, the branching curves are  strictly abnormal. 

\subsection{First example}\label{sec:exp1}
The example is given by a rank-two sub-Riemannian structure on $\R^4$ with a region where the step is four and one where the step is three. In accordance with Proposition \ref{prop:analytic}, the structure is smooth but not real-analytic. The choice of working in $\R^4$ is motivated by the fact that a rank-two sub-Riemannian structure of step four with $\Delta_2\neq \Delta_3$ and $\Delta_3^\perp \neq \{0\}$ must have dimension at least equal to $4$. In the example, we consider a horizontal curve $\gamma$ which stays in the  step-four region before time $\tau>0$ and enters the step-three region after time $\tau$. In this example, $\gamma|_{[0,\tau]}$ has a two-parameter family of abnormal lifts, generated by a fixed regular abnormal lift $\lambda(\cdot)$ of $\gamma$ and a totally non-regular lift in $\Delta_3^\perp$. In particular $\tnr(\tau)$ is nontrivial, while for $t > \tau$, $\Delta_3^\perp(\gamma(t))$ becomes trivial, and also $\tnr(t)$ as a consequence. Namely,
\begin{equation}
\label{eq:Aex1}
\mathcal{A}_t =  \spn\{\lambda(0)\} \oplus (\tnr(t)\sqcup \{0\}) \quad \forall t, \quad \tnr(\tau) \neq \emptyset, \quad \tnr(t)= \emptyset \quad \forall t >\tau.
\end{equation}

By Theorem \ref{th:abn_br}, $\gamma$ must have a regular abnormal branching at time $\tau$. Since in the example $\gamma|_{[0,\tau]}$ is strictly abnormal, then all geodesics branching from it are strictly abnormal as well thanks to Proposition \ref{prop:ex1}. In building our example, we were inspired by a strategy given in \cite{rifford2014sub} to produce strict abnormals, that we briefly recall. Consider the Martinet-type distribution on $\R^3$ generated by
\begin{equation}
    X_1=\partial_1, \quad X_2=(1+x_1\phi(x))\partial_2 + x_1^2\partial_3,
\end{equation}
where $\phi$ is a smooth function on $\R^3$. This is simply a perturbation of the flat Martinet distribution, which   corresponds to the case $\phi\equiv 0$. Let $\gamma(t)=(0,t,0)$ for $t\in\R$. It is proved in \cite{rifford2014sub} that, if $\phi(\gamma(t))\neq 0$ at some $t\in \R$, then there exists $\varepsilon>0$ such that $\gamma|_{[t,t+\varepsilon]}$  is a strictly abnormal length minimizer. 
\\

Let us turn to the example. Consider the step-four sub-Riemannian structure on $\R^4$ defined by the global orthonormal frame
\begin{equation}
\label{eq:A1}
    X_1 = \partial_1, \quad X_2= (1+x_1x_2)\partial_2 + \frac{x_1^2}{2}\partial_3 + \frac{x_1^3}{6}\partial_4 
    \tag{A1}
\end{equation}
and the step-three one defined by
\begin{equation}
\label{eq:B1}
    X_1 = \partial_1, \quad X_2= (1+x_1x_2)\partial_2 + \frac{x_1^2}{2}\partial_3 + x_1x_2\partial_4.
    \tag{B1}
\end{equation}
Let $\theta:\R \to [0,1]$ be a smooth transition function such that $\theta(y)=0$ if $y \le 0$, $\theta(y)=1$ if $y \ge 1$, and $\theta'(y)>0$ for $0<y<1$, and consider the distribution on $\R^4$ generated by
\begin{equation}
\label{eq:T1}
    X_1=\partial_1, \quad X_2= (1+x_1x_2)\partial_2 + \frac{x_1^2}{2}\partial_3 +x_1A(x_1,x_2)\partial_4,
\tag{T1}
\end{equation}
where 
\begin{equation}
    A(x_1,x_2)=(1-\theta(x_2))\frac{x_1^2}{6}+\theta(x_2)x_2.
\end{equation}

The distribution generated by \eqref{eq:T1} is simply the transition between the ones generated by \eqref{eq:A1} on $\{x_2\leq 0\}$ and \eqref{eq:B1} on $\{x_2 \geq 1\}$.  Fix $T>\tau>0$ and consider the horizontal curve
\begin{equation}
    \gamma:[0,T]\to \R^4, \quad \gamma(t)=(0,t-\tau,0,0).
\end{equation}
Notice that $\gamma$ has unit speed and, as anticipated, $\gamma(t)$ belongs to a region of step four for $t\le \tau$ and step three for $t > \tau$. Indeed, an explicit computation shows that
\begin{align}
    X_2(\gamma(t)) &=\partial_2, &  X_{12}(\gamma(t))& = (t-\tau)\partial_2 + \theta(t-\tau)(t-\tau)\partial_4, \\
    X_{112}(\gamma(t))& = \partial_3, & X_{212}(\gamma(t))& = \partial_2 + (\theta'(t-\tau)(t-\tau)+ \theta(t-\tau))\partial_4.
\end{align}

\begin{proposition}\label{prop:strictabn}
    The curve 
    $\gamma|_{[0,\tau]}$ has no normal lifts.
\end{proposition}

\begin{proof}
    Since $\gamma|_{[0,\tau]}$ is contained in a region where \eqref{eq:A1} and \eqref{eq:T1} coincide, we work only with the distribution \eqref{eq:A1}. Assume by contradiction that there exists a normal lift, that we write in coordinates $\lambda(t)=(x(t),p(t))=((0,t-\tau,0,0),(p_1(t),p_2(t),p_3(t),p_4(t)))$. This lift solves the Hamiltonian system associated with the sub-Riemannian Hamiltonian defined by $2H(x,p)=h_1(x,p)^2+h_2(x,p)^2$, where
    \begin{equation}
        h_1(x,p)= p_1, \quad h_2(x,p)= (1+x_1x_2)p_2 + \frac{x_1^2}{2}p_3 + \frac{x_1^3}{6}p_4.
    \end{equation}
    The Hamiltonian system reads
    \begin{equation}
    \label{eq:systex1}
        \begin{cases}
            \dot{x}_1 = p_1, \\
            \dot{x}_2 =h_2(x,p)(1+x_1x_2), \\
            \dot{x}_3 =h_2(x,p)\frac{x_1^2}{2}, \\
            \dot{x}_4 =h_2(x,p)\frac{x_1^3}{6}, \\
            \dot{p}_1=-h_2(x,p)\left(x_2p_2+x_1p_3+\frac{x_1^2}{2}p_4 \right), \\
            \dot{p}_2=-h_2(x,p)x_1p_2, \\
            \dot{p}_3=\dot{p}_4=0.
        \end{cases}
    \end{equation}
    Letting $\lambda(0)=((0,-\tau,0,0),(\bar{p}_1,\bar{p}_2,\bar{p}_3,\bar{p}_4))$ we have $p_3(t)\equiv \bar{p}_3$, $p_4(t) \equiv \bar{p}_4$, and 
    \begin{equation}
        H(x(t),p(t))=H(x(0),p(0))=\frac{1}{2}\left(\bar{p}_1^2+\bar{p}_2^2 \right) \neq 0 \quad \forall t \in [0,\tau]. 
    \end{equation}
    Since $x_1(t)\equiv0$ we deduce that $p_1(t) \equiv \bar{p}_1= 0$ from the first equation. Therefore since the Hamiltonian is constant we have $p_2(t)\equiv \bar{p}_2 \neq 0$. Finally, the fifth equation gives $0 \equiv -\bar{p}_2^2 x_2(t)$. However,  $x_2(0)=-\tau \neq 0$, which implies $\bar{p}_2 =0$, leading to a contradiction.
\end{proof}

We prove \eqref{eq:Aex1}. Abnormal lifts of $\gamma$ are solutions of $\dot{\lambda}(t) = \vec{h}_2(\lambda(t))$. Writing $\lambda(t)=(x(t),p(t))$ with $\lambda(0)=((0,-\tau,0,0),(\bar{p}_1,\bar{p}_2,\bar{p}_3,\bar{p}_4))$ we find that $p(t)$ satisfies
\begin{align}
    \begin{cases}
    \dot{p}_1(t) = - (t-\tau)\big(p_2(t) + \theta(t-\tau)p_4(t)\big), \\
    \dot{p}_2(t) =\dot{p}_3(t)=\dot{p}_4(t)=0,   
    \end{cases}
\end{align}
hence $(p_2(t),p_3(t),p_4(t)) \equiv (\bar{p}_2,\bar{p}_3,\bar{p}_4)$. Imposing that $\lambda(t) \in \Delta^\perp(\gamma(t))$ and recalling that $X_1=\partial_1$ and $X_2 = \partial_2$ along $\gamma$, we find 
 $p_1(t)\equiv 0 \equiv p_2(t)$, and then
\begin{equation}
\label{eq:p1}
    0= p_1(t) = -\bar{p}_4\int_0^t (s-\tau)\theta(s-\tau)\,ds.
\end{equation}
Depending on the restriction 
of $\gamma$ that we consider, we may impose some additional constraints on the initial covector. We distinguish between two cases:
\begin{itemize}
    \item If we consider 
    $\gamma|_{[0,t]}$ with $t\le \tau$ by the definition of $\theta$ we have 
     that \eqref{eq:p1} is satisfied regardless of the choice of $\bar{p}_4$ and we find a two-parameter family of abnormal lifts, namely,
        \begin{equation}
            \lambda(s)=\big( (0,s-\tau , 0,0) , (0,0,\bar{p}_3, \bar{p}_4)\big ) , \quad (\bar{p}_3,\bar{p}_4) \in \R^2 \setminus \{(0,0)\}.
        \end{equation}
        Notice that these lifts are regular abnormal  if and only if $\bar{p}_3 \neq 0$, while, omitting the base-point from the notation of covectors, it holds
                \begin{equation}
            \tnr(t) =  \{(0,0,0,\bar{p}_4)\mid \bar{p}_4 \neq 0\} \quad \forall t\le \tau.
        \end{equation}
    \item If we consider 
    $\gamma|_{[0,t]}$ with $t > \tau$, the integral in \eqref{eq:p1} becomes strictly positive after time $\tau$, and this imposes $\bar{p}_4=0$. Therefore the segment $\gamma|_{[0,t]}$ has only a one-parameter family of abnormal lifts that are also regular, which is 
        \begin{equation}
            \lambda(s)= \big( (0,s-\tau , 0,0) , (0,0,\bar{p}_3,0)\big ), \quad \bar{p}_3 \neq 0. 
        \end{equation}
        In particular $\tnr(t)$ is trivial as soon as $t>\tau$.
\end{itemize}
        
Summarizing, for $t \le \tau$ we have
\begin{equation}
    \mathcal{A}_t = \mathrm{REG}(t) \sqcup \tnr(t) \sqcup \{0\},
\end{equation}
with $\mathrm{REG}(t)= \{(0,0,\bar{p}_3,\bar{p}_4) \mid \bar{p}_3 \neq 0,\, \bar{p}_4 \in \R\}$ and  $ \tnr(t)=\{(0,0,0,\bar{p}_4)\mid \bar{p}_4 \neq 0\}$, while for $t>\tau$ it holds 
\begin{equation}
    \mathcal{A}_t= \mathrm{REG}(t) \sqcup \{0\},
\end{equation}
with $\mathrm{REG}(t)=\{(0,0,\bar{p}_3,0) \mid \bar{p}_3 \neq 0\}$ and $\tnr(t)= \emptyset$.  This proves \eqref{eq:Aex1} and the fact that $\gamma$ has a regular abnormal branching at time $\tau$  by Theorem \ref{th:abn_br}. Furthermore, by Proposition \ref{prop:strictabn} and Proposition \ref{prop:ex1}, all branching geodesics are strictly abnormal.

\begin{figure}[H]
     \centering
     \includegraphics[width=200pt]{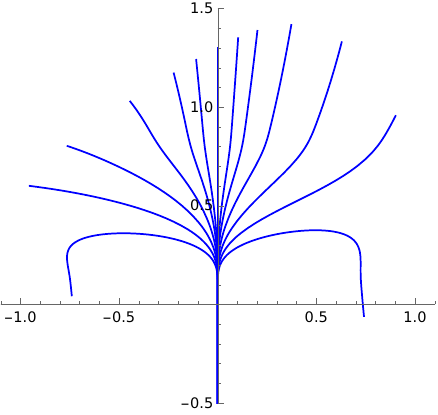}
     \caption{Projection on the $(x_1,x_2)$ plane of the regular abnormal geodesics branching from $\gamma$.}
\end{figure}

Recall the description of the regular abnormal branching given in Section \ref{sect:br_descr}. Let us fix $\lambda_0:= ((0,-\tau,0,0),(0,0,1,0)) \in T_{\gamma(0)}^*\R^4$  and define 
\begin{equation}
    \xi_\alpha = \big ((0,-\tau,0,0), (0,0,0,\alpha) \big), \quad  \lambda_\alpha = \lambda_0+\xi_\alpha, \quad \alpha \neq 0.
\end{equation}
We have that
\begin{equation}
    \tnr(\tau)= \{ \xi_\alpha \mid \alpha \neq 0  \}.
\end{equation}
Identifying $\tnr(\tau) \sqcup \{0\}$ and  $\R$, since $\tnr(t) = \emptyset$ for $t > \tau$ the family of geodesics that branch from $\gamma$ at time $\tau$ is given, with a slight abuse of notation, by the map $\Phi : ([0,T] \times \R) \cap \mathcal{D}(\lambda_0) \to \R^4$ defined as
\begin{equation}
    \Phi(t,\alpha)= \pi\circ e^{t\vec{G}}(\lambda_\alpha). 
\end{equation}

We now prove that this map embeds a $2$-dimensional surface locally around the $x_2$-axis. The following proposition is the analog of \cite[Proposition 11]{Mietton-Rizzi}.


\begin{proposition}\label{prop:embedding}
    The above defined map $\Phi$ is an embedding in a neighborhood of any point $(t,0)$ with $\tau<t<T$.
\end{proposition}

\begin{proof}
    The equality $\Phi(t,0)=(0,t-\tau,0,0)$ implies that $\frac{\partial\Phi}{\partial t}(t,0)=\partial_2$. We prove that $\frac{\partial\Phi}{\partial \alpha}(t,0)$ has nonzero component along $\partial_1$ for $t>\tau$. Denote $e^{t\vec{G}}(\lambda_\alpha)=(x^\alpha(t),p^\alpha(t))$. We have
    \begin{equation}
        x_1^\alpha(t) = - \int_\tau^t \frac{h_{212}(x^\alpha(s),p^\alpha(s))}{|\h(x^\alpha(s),p^\alpha(s))|} ds. 
    \end{equation}
    An explicit computation shows that
    \begin{align}
        h_{112}(x,p)& = p_3 + p_4x_1(1-\theta(x_2)) \\
        h_{212}(x,p) &= p_2 + p_4 \left( \theta'(x_2) \left(-\frac{x_1^2}{2} - \frac{x_1^3x_2}{3} +x_2 \right) +\theta(x_2) \right),
    \end{align}
    and combining this with the observation that the $p_3,p_4$ components of the covectors are constant along the flow of $\vec{G}$ we find that
    \begin{equation}
        \frac{h_{212}(x^\alpha(t),p^\alpha(t))}{|\h(x^\alpha(t),p^\alpha(t))|} = \alpha f(t,\alpha),
    \end{equation}
    where $f$ is a smooth function of $(t,\alpha)$ when $\alpha$ is small and 
    \begin{equation}
        f(t,0)= \theta'(t-\tau)(t-\tau)+\theta(t-\tau) > 0 \quad \text{when  } t>\tau.
    \end{equation}
    Therefore,
    \begin{equation}
        \left.\frac{\partial x_1^\alpha(t)}{\partial \alpha} \right |_{\alpha=0} = -\int_0^t (\theta'(s)s+\theta(s)) \ ds \ <0 \quad \text{for } t>\tau. 
        \qedhere
    \end{equation} 
\end{proof}


\subsection{Second example}\label{sec:exp2}
Consider the step-four distribution on $\R^4$ generated by
\begin{equation}
\label{eq:A2}
    X_1=\partial_1, \quad X_2=\partial_2 + \frac{x_1^2}{2}\partial_3+ \frac{x_1^3}{6}\partial_4,
    \tag{A2}
\end{equation}
and the step-three one generated by
\begin{equation}
\label{eq:B2}
    X_1=\partial_1, \quad X_2=\partial_2 + \frac{x_1^2}{2}\partial_3+ x_1x_2\partial_4.
    \tag{B2}
\end{equation}
Let $\theta$ be a smooth transition function as in the previous example. Consider the sub-Riemannian structure on $\R^4$ defined by the global orthonormal frame
\begin{equation}
\label{eq:T2}
    X_1=\partial_1, \quad X_2= \partial_2 + \frac{x_1^2}{2}\partial_3+ x_1A(x_1,x_2)\partial_4,
    \tag{T2}
\end{equation}
where 
\begin{equation}
    A(x_1,x_2)=(1-\theta(x_2))\frac{x_1^2}{6}+ \theta(x_2)x_2.
\end{equation}
Again, this distribution is simply the transition between the distributions generated by \eqref{eq:A2} on $x_2\leq 0$ and \eqref{eq:B2} on $x_2\geq 1$. Fix $T>\tau>0$ and consider the curve
\begin{equation}
    \gamma:[0,T] \to \R^4, \quad \gamma(t)=(0,t-\tau,0,0).
\end{equation}
Along $\gamma$ we compute
\begin{align}
    X_2(\gamma(t)) & =\partial_2, \quad X_{12}(\gamma(t)) = \theta(t-\tau)(t-\tau)\partial_4, \\
    X_{112}(\gamma(t)) &= \partial_3, \quad X_{212}(\gamma(t)) = (\theta'(t-\tau)(t-\tau)+ \theta(t-\tau))\partial_4,
\end{align}
thus $\gamma(t)$ lies in a region where the step is $4$ for $t \le \tau$ and where the step is $3$ for $t>\tau$. \\

An explicit computation shows that for $t \le \tau$ the restriction $\gamma|_{[0,t]}$ has a two-parameter family of normal lifts, given in coordinates by
\begin{equation}
    \lambda (s) = (x(s),p(s))= \big( (0,s-\tau,0,0), (0,1,\bar{p}_3,\bar{p}_4)), \quad (\bar{p}_3,\bar{p}_4) \in \R^2,
\end{equation}
which corresponds to an affine plane $\mathcal{N}_t$ in $T^*_{\gamma(0)}M$. As soon as $t>\tau$, the space of normal lifts loses dimension, restricting to the one-parameter family
\begin{equation}
    \lambda (s) = (x(s),p(s))= \big( (0,s-\tau,0,0), (0,1,\bar{p}_3,0)), \quad \bar{p}_3 \in \R.
\end{equation}

Let us turn to the abnormal lifts of $\gamma$. 
Again, they are solutions of $\dot{\lambda}(t) = \vec{h}_2(\lambda(t))$, and writing $\lambda(t)=(x(t),p(t))$ with $\lambda(0)=((0,-\tau,0,0),(\bar{p}_1,\bar{p}_2,\bar{p}_3,\bar{p}_4))$, we find
\begin{align}
    \begin{cases}
        \dot{p}_1(t)= -\theta(t-\tau)(t-\tau)p_4(t), \\
        \dot{p}_2(t)=\dot{p}_3(t)=\dot{p}_4(t)=0,
    \end{cases}
\end{align}
hence $(p_2(t),p_3(t),p_4(t)) \equiv (\bar{p}_2,\bar{p}_3,\bar{p}_4)$. Imposing $\lambda(t) \in \Delta^\perp(\gamma(t))$ we find that 
 $p_1(t)\equiv 0 \equiv p_2(t)$, and 
\begin{equation}
   0= p_1(t) = - \bar{p}_4 \int_0^t \theta(s-\tau)(s-\tau) \, ds.
\end{equation}
Again, we distinguish between two cases:
\begin{itemize}
    \item If we consider 
    $\gamma|_{[0,t]}$ with $t < \tau$, then $p_1 \equiv0$ on $[0,t]$ without further restrictions on the initial covector. The curve $\gamma|_{[0,t]}$ has a two-parameter family of abnormal lifts, namely,
        \begin{equation}
            \lambda(s) = \big( (0,s-\tau,0,0), (0,0,\bar{p}_3,\bar{p}_4) \big), \quad (\bar{p}_3,\bar{p}_4) \in \R^2 \setminus \{0\},
        \end{equation}
        which are regular if $\bar{p}_3 \neq 0$, with $\tnr(t) =  \{(0,0,0,\bar{p}_4)\mid \bar{p}_4 \neq 0\}$. 
    \item If we consider 
    $\gamma|_{[0,t]}$ with $t > \tau$, we need also to impose $\bar{p}_4=0$. In this case, there is only a one-parameter family of abnormal lifts, that are also regular, namely,
        \begin{equation}
            \lambda(s) = \big( (0,s-\tau,0,0), (0,0,\bar{p}_3,0) \big), \quad \bar{p}_3 \neq 0,
        \end{equation}
        and $\tnr(t) = \emptyset$.
\end{itemize}
Notice that this is coherent with the drop of dimension of the space of normal lifts. The decomposition of $\mathcal{A}_t$ is the same as in the previous example. Again, \eqref{eq:Aex1} holds and $\gamma$ has both a normal and a regular abnormal branching at time $\tau$ (indeed, both Theorems \ref{th:abn_br} and \ref{th:n_br} apply). By Proposition \ref{prop:strictn}, the regular abnormal geodesics that branch from $\gamma$ at time $\tau$ are strictly abnormal. \\

 \begin{figure}[H]
     \centering
     \includegraphics[width=170pt]{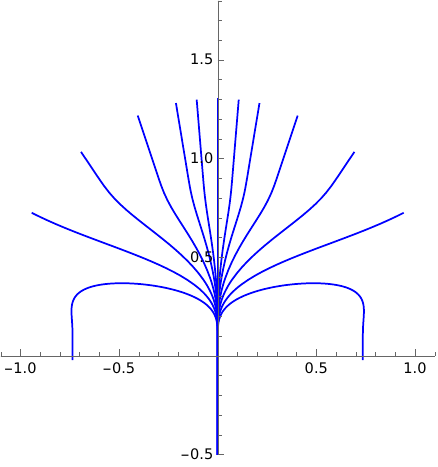} \quad
     \includegraphics[width=170pt]{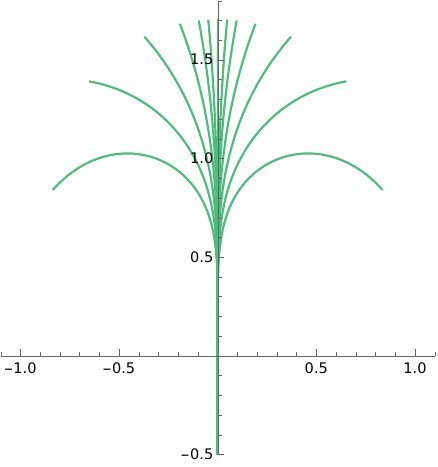}
     \caption{Projection on the $(x_1,x_2)$ plane of the regular (strictly) abnormal 
     (blue, left) and normal geodesics (green, right) branching from $\gamma$.}
     \label{fig:placeholder}
 \end{figure}

Defining 
 $\Phi$ as in the previous example (see Proposition \ref{prop:embedding}), it is possible to prove that it is a local embedding at any point $(t,0)$ with $\tau<t<T$.

\section{Local branching}
\label{sect:local_branching}
In the study of branching of geodesics in sub-Riemannian manifolds it is possible to give the following definition of branching, which can be applied to geodesics in general and does not refer to the existence of lifts. It can be also taken as a definition of branching in the more general class of metric spaces.
\begin{definition}
\label{def:branching_gen}
    A geodesic $\gamma:[0,T]\to M$ \emph{branches} at time $\tau \in (0,T)$ if there exist $\varepsilon_0 >0$ and a geodesic $\bar{\gamma}:[0,\tau+\varepsilon_0] \to M$ such that $\bar{\gamma} |_{[0,\tau]} = \gamma|_{[0,\tau]}$ and $\bar{\gamma} |_{[0,\tau+\varepsilon]} \neq \gamma|_{[0,\tau+\varepsilon]}$  for all $0< \varepsilon \le \varepsilon_0$.
\end{definition}
Let us introduce a notion of branching that sits in between Definition \ref{def:branching_gen} and the Definitions \ref{def:n_br} and \ref{def:ra_br} of normal and regular abnormal branching.

This new approach is motivated by two observations. The first one is that a global definition of the vector field $\vec{G}$ (and therefore of regular abnormal paths) requires the existence of a global orthonormal frame for the distribution. Such a basis does not exist, in general, not even on a neighborhood of a fixed geodesic $\gamma$, but it is always possible to find one if we reduce to a neighborhood of a sufficiently short subsegment of $\gamma$. The second observation is that, as detailed later in this section, with techniques similar to those used in \cite{Mietton-Rizzi} and in Sections \ref{sec:exp1} and \ref{sec:exp2}, it is possible to produce examples of branching (in the sense of Definition \ref{def:branching_gen}) of normal and regular abnormal geodesics that rely only on information on a subsegment of $\gamma$ and cannot be detected via Theorem \ref{th:n_br} and Theorem \ref{th:abn_br} (even when a global orthonormal frame exists). \\

Let us focus on this second motivation. In the following we assume the setting of Section~\ref{sec:prel}, namely that of a sub-Riemannian structure of rank $2$ on a smooth manifold $M$, admitting for simplicity a global frame, which allows us to use Definition~\ref{def:reg_abn} of regular abnormals.
\begin{definition}
\label{def:loc_br}
    A geodesic $\gamma:[0,T]\to M$ has a \emph{local normal (resp.\ regular abnormal) branching} at time $\tau \in (0,T)$ if there exist an interval $[t_1,t_2] \subset [0,T]$ with $\tau \in (t_1,t_2)$ and a normal (resp.\ regular abnormal) geodesic $\tilde{\gamma}:[t_1,t_2] \to M$ that branches from $\gamma|_{[t_1,t_2]}$ at time $\tau$.
\end{definition}

 Notice that we do not require $\gamma|_{[t_1,t_2]}$ to be normal or regular abnormal. When  $\gamma|_{[t_1,t_2]}$ is normal (resp.\ regular abnormal), the branching of $\tilde{\gamma}$ from $\gamma|_{[t_1,t_2]}$ is characterized by Theorem \ref{th:n_br} (resp.\ Theorem \ref{th:abn_br}). Observe that if $\gamma$ has a local normal (resp.\ regular abnormal) branching at time $\tau$, we can consider a geodesic $\tilde{\gamma}$ that branches from  $\gamma|_{[t_1,t_2]}$ at time $\tau$ and denote as $\bar{\gamma} = \gamma|_{[0,t_1]} \ast \tilde{\gamma}$ the concatenation
\begin{align}
    \bar{\gamma}(t) = 
    \begin{cases}
        \gamma(t) & \text{if } 0 \le t \le t_1, \\
        \tilde{\gamma}(t) & \text{if } t_1<t\le t_2.
    \end{cases}
\end{align}
The curve $\bar{\gamma}$ is a geodesic that branches from $\gamma$ at time $\tau$, in the sense of Definition \ref{def:branching_gen}, but it does not need to be a normal or regular abnormal geodesic; in particular, such branching is not detected by Theorem \ref{th:n_br} (resp.\ Theorem \ref{th:abn_br}). We now illustrate this fact with an example.  \\

The example is inspired by the one  in \cite{Mietton-Rizzi}, that we briefly recall. On $\R^3$ we consider the sub-Riemannian structure generated by the global orthonormal frame
\begin{equation}
\label{eq:norm_trans}
    X_1=\partial_1, \quad X_2=\partial_2 + x_1 A(x_1,x_2)\partial_3,
\end{equation}
where 
\begin{equation}\label{eq:Ax1x2}
    A(x_1,x_2)=(1-\theta(x_2))\frac{x_1}{2} + \theta(x_2).
\end{equation}
and $\theta:\R\to [0,1]$ is a smooth transition function such that $\theta(y)=0$ if $y \le 0$, $\theta(y)=1$ if $y \ge 1$, and $\theta'(y)>0$ for $0<y<1$. Such distribution coincides with the flat Martinet distribution on $\{x_2\le 0\}$ and the Heisenberg distribution on $\{x_2\ge1\}$. Consider the horizontal curve 
\begin{equation}
\label{eq:gamma}
    \gamma:[-1,1] \to \R^3, \quad \gamma(t)=(0,t,0).
\end{equation}
The Hamilonian system associated with \eqref{eq:norm_trans} is
\begin{align}
\label{eq:syst}
    \begin{cases}
        \dot{x}_1=p_1 \\
        \dot{x}_2=h_2(x,p) \\
        \dot{x}_3=h_2(x,p)x_1A(x_1x_2) \\
        \dot{p}_1=-h_2(x,p)((1-\theta(x_2))x_1+\theta(x_2))p_3\\
        \dot{p}_2=-h_2(x,p)\theta'(x_2)(-\frac{x_1^2}{2}+x_1)p_3\\
        \dot{p}_3=0,
    \end{cases}
\end{align}
with $h_2(x,p)=p_2+x_1A(x_1,x_2)p_3$, hence for $t\le 0$, the subsegments $\gamma|_{[-1,t]}$ have a one-parameter family of normal lifts, given by
\begin{equation}
\label{eq:lift_n}
    \lambda(s) = ((0,s,0),(0,1,\bar{p}_3)), \quad \bar{p}_3 \in \R, 
\end{equation}
while for $t>0$, the segment $\gamma|_{[-1,t]}$ possesses a unique normal lift that corresponds to the case $\bar{p}_3=0$. 
Following Theorem~\ref{th:n_br}, $\gamma$ has a normal branching at time $\tau=0$, and no other normal branching can occur at later times, since the abnormal vector spaces $\mathcal{A}_t$ (which, with a slight abuse of notation, are associated with the curves $\gamma|_{[-1,t]}$)   become trivial as $t>0$. \\

Let us now consider the distribution \eqref{eq:norm_trans}, with $A(x_1,x_2)$ as in \eqref{eq:Ax1x2}, where $\theta$ is chosen as follows. First we consider a smooth function $\theta_1:\R \to [0,1]$ such that
\begin{itemize}
    \item $\theta_1(y)=0$ for $y \in (-\infty, -\frac{4}{5}]\cup [-\frac{1}{5},+\infty)$,
    \item $\theta_1(y)=1$ for $y \in \left[-\frac{3}{5},-\frac{2}{5}\right]$,
    \item $\theta_1(y) \in (0,1)$ with $\theta'(y)>0$ for $y \in (-\frac{4}{5},-\frac{3}{5})$ and $\theta_1(y) \in (0,1)$ with $\theta'(y)<0$ for $y \in (-\frac{2}{5},-\frac{1}{5})$. 
\end{itemize}
For all $n \ge 1$ let $\theta_n:\R \to [0,1]$ be defined by $\theta_n(y):=\theta_1(5^{n-1}y)$ and set $\bar{\theta}(y):\R\to [0,1]$, $\bar{\theta}(y):=\sum_{n \ge 1} \theta_n(y)$. Finally we let
\begin{align}
    \theta:\R\to[0,1], \quad \theta(y):= 
    \begin{cases}
        \bar{\theta}(y) e^{1/y} & \text{if } y \neq 0, \\
        0 & \text{if } y=0.
    \end{cases}
\end{align}
 \begin{figure}[H]
     \centering
     \includegraphics[width=200pt]{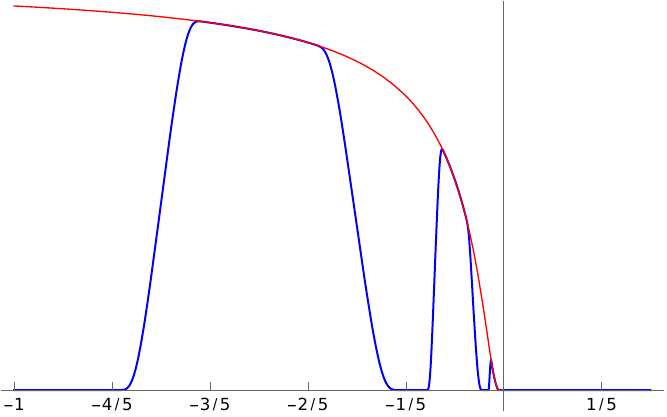} 
     \caption{Plot of the functions $\theta(y)$ (blue) on $\R$ and $e^{1/y}$ (red).}
     \label{fig:placeholder}
 \end{figure} 
Since $\theta$ is smooth, the distribution \eqref{eq:norm_trans} is smooth; moreover, it is totally nonholonomic. Define the horizontal curve
\begin{equation}
    \gamma:[-1,1] \to \R^3, \quad \gamma(t) = (0,t,0).
\end{equation}
The curve $\gamma$ is normal, with only one normal lift. However, given $-1 \le t_1 < t_2 \le 1$, considering the solutions of system \eqref{eq:syst} that project on $\gamma|_{[t_1,t_2]}$ we infer that the restriction $\gamma|_{[t_1,t_2]}$ has a one-parameter family of normal lifts as in \eqref{eq:lift_n} if and only if $\theta \equiv 0$ on $[t_1,t_2]$. In particular, for each $n \ge 1$ and $t>-1/5^{n-1}$ the restriction of $\gamma$ to the interval $[-1/5^{n-1},t]$ undergoes a loss of dimension of the normal affine space at $t= \tau_n :=-4/5^n $. Therefore the geodesic $\gamma$ has a local normal branching at $t=\tau_n$. Notice that the branching times are accumulating at $t=0$. 

This new phenomenon is not in contradiction with the fact that a normal geodesic has at most a finite number of normal branchings, \emph{in the sense of Definition \ref{def:n_br}}. Indeed, let $n \ge 2$ and consider a geodesic 
\begin{equation}
    \bar{\gamma}=\gamma|_{[-1,-1/5^{n-1}]} \ast \tilde{\gamma},
\end{equation} 
where $\tilde{\gamma}$ is a normal geodesic that branches from $\gamma|_{[-1/5^{n-1}, -3/5^n]}$  (in the sense of Definition \ref{def:n_br}) at time $\tau_n$. If $\bar{\gamma}$ were normal, since its restriction to the interval $(-4/5^{n-1},-1/5^{n-1})$ is contained in a region where the distribution is of contact type, the normal lift of $\bar{\gamma}$ would be unique. Moreover, since $\bar{\gamma}$ coincides with $\gamma$ up to time $\tau_n$, such a lift should coincide with the unique normal lift of $\gamma$, which would prevent branching. This proves that $\bar{\gamma}$ has no normal lift.\\

With analogous techniques, starting from examples of Sections \ref{sec:exp1} and \ref{sec:exp2}, it is possible to produce occurrences of regular abnormal geodesics that branch an arbitrarily large (even infinite) number of times, whose branching cannot all be detected via Theorem \ref{th:abn_br}.  Let us finally provide a necessary condition for a normal geodesic to have a local regular abnormal branching. 

\begin{proposition}
    Let $\gamma:[0,T]\to M$ be a unit speed normal geodesic that has a local regular abnormal branching at time $\tau \in (0,T)$. Then $\gamma$ has a local normal branching at time $\tau$.
\end{proposition}
\begin{proof}
    Without loss of generality, adopting the notation of Definition \ref{def:loc_br}, assume that $t_1=0$ and $t_2=T$.  Notice that in this setting (Definition \ref{def:loc_br}) we are not assuming that $\gamma|_{[0,t]}$ has a regular abnormal lift for $t>\tau$, therefore Theorem \ref{th:abn_br} cannot be applied. Let $\lambda(\cdot)$ be a normal lift of $\gamma$ and $\tilde{\mu}(\cdot)$ a regular abnormal lift of $\tilde{\gamma}$. The sum of $\lambda|_{[0,\tau]}$ and $\tilde{\mu}|_{[0,\tau]}$ is a normal lift of $\gamma|_{[0,\tau]}$, and it can be extended on the interval $[0,\tau+\varepsilon]$ for some $\varepsilon>0$ as an integral line $\xi(\cdot)$ of $\vec{H}$. By contradiction, if $\xi(\cdot)$ were a normal lift of $\gamma|_{[0,t+\varepsilon']}$ for some $0<\varepsilon'\le \varepsilon$, then $\xi(\cdot)-\lambda(\cdot)$ would be an abnormal lift of the same curve, with $(\xi-\lambda)|_{[0,\tau]}=\tilde{\mu}|_{[0,\tau]}$. Observing that $\tilde{\mu}(\tau) \in \mathcal{R}$ and thanks to Proposition \ref{prop:Imax}, up to possibly reducing $\varepsilon'$ we would deduce that $(\xi- \lambda)|_{[0,\tau+\varepsilon^{'}]}=\tilde{\mu}|_{[0,\tau+\varepsilon^{'}]}$, which would lead to a contradiction,  since their projections do not coincide on the interval $[0,\tau+\varepsilon]$. 
\end{proof}

\section{Optimality of short arcs}
\label{sect:opt_arcs}
Our examples of branching of regular abnormal geodesics present some similarities with the branching of normal ones. One of their features is that they give rise to families of branching geodesics, which depend smoothly on a finite-dimensional set of parameters. A natural question is the following. 
\begin{question}
    Let $\gamma:[0,T]\to M$ be a normal (resp.\ regular abnormal) geodesic and $(\gamma^\lambda)_{\lambda \in \Lambda}$ a family of normal (resp.\ regular abnormal) geodesics that branch from $\gamma$ at time $\tau \in(0,T)$. Does it exist  $\rho > 0$ such that the restrictions to the interval $[\tau-\rho,\tau+\rho]$ of each geodesic $(\gamma^\lambda)_{\lambda \in \Lambda}$, is a length-minimizer?
\end{question} 
If $\Lambda$ has finite cardinality, the answer is clearly positive. As proven in \cite{Mietton-Rizzi}, the answer is again positive for normal geodesics, if we consider compact families of branching geodesics. In this section, we prove that the same result holds true for regular abnormal geodesics. This result will be obtained as a corollary of Theorem~\ref{thm:opt_cpt}, which generalizes Theorem~\ref{pr:liu-sus}.  

\begin{notation}
     Recall the definition of $\mathcal{R}_T$ (Definition \ref{def:rt}). For all $\lambda\in\mathcal{R}_T$ we let $\gamma^\lambda:[0,T]\to M$ be the regular abnormal geodesic $\gamma^\lambda(t)=\pi\circ e^{t\vec{G}}(\lambda)$.
\end{notation}

\begin{theorem}
\label{thm:opt_cpt}
    Let $T>0$. For all $\tau\in [0,T]$ and $\lambda_0\in\mathcal{R}_T$ there exist $\rho>0$ and an open neighborhood $\Lambda$ of $\lambda_0$ in $\mathcal{R}_T$ such that for all $\lambda\in \Lambda$, the restriction 
    \begin{equation}
        \gamma^\lambda|_{[\tau-\rho,\tau+\rho]\cap [0,T]}
    \end{equation} 
    is the unique length-minimizing curve, parametrized with unit speed, between its endpoints.
\end{theorem}

\begin{corollary}
\label{cor:cpt_br}
    Let $\gamma:[0,T]\to M$ be a regular abnormal geodesic that has regular abnormal branching at time $\tau \in (0,T)$. Let $K \subset \mathrm{REG}^+(\tau)$ be compact. There exists $\rho>0$ such that for all $\lambda \in K$ the curve $\gamma^\lambda|_{[\tau-\rho,\tau+\rho]\cap[0,T]}$ is well-defined and it is the unique length-minimizer, parametrized with unit speed, between its endpoints. 
\end{corollary}
\begin{proof}[Proof of Corollary \ref{cor:cpt_br}]
      By compactness of $K$ there exists $\varepsilon>0$ such that $\gamma^\lambda$ is well-defined on $[0,\tau+\varepsilon]$ for all $\lambda \in K$. By Theorem \ref{thm:opt_cpt}, for all $\lambda \in K$ we can find an open neighborhood $V(\lambda)$ of $\lambda$ in $\mathcal{R}_T$ and $\rho(\lambda)>0$ (which can be supposed to be $\le \varepsilon$) such that for all $\mu \in V(\lambda)$ the geodesic $\gamma^\mu|_{[\tau-\rho(\lambda),\tau+\rho(\lambda)]\cap [0,T]}$ is uniquely length-minimizing. Since $K$ is compact, we can find $\lambda_1,\dots,\lambda_N \in K$ such that $K \subset \bigcup_{i=1}^NV(\lambda_i)$. It is now sufficient to let $\rho:= \min \{\rho(\lambda_1),\dots,\rho(\lambda_N)\}$.
\end{proof}

The strategy to prove Theorem \ref{thm:opt_cpt} is the following. Fix $T,\tau$, and $\lambda_0$ as above, letting $q_0:= \gamma^{\lambda_0}(\tau)$. First we show that we can choose an open neighborhood $U$ of $q_0$, an open neighborhood $\Lambda$ of $\lambda_0$ in $\mathcal{R}_T$ with compact closure, and $\varepsilon>0$ so that the curves $\gamma^\lambda|_{[\tau-\varepsilon, \tau+\varepsilon] \cap [0,T]}$, for $\lambda \in \Lambda$, are contained in $U$ and can be simultaneously put in a normal form via a family of coordinate changes on $U$ that depend smoothly on $\lambda$. Then we will prove that, up to further restricting $\Lambda$, $U$, and $\varepsilon$, there exists $0<\rho\le \varepsilon$ such that restrictions to the interval $[\tau-\rho, \tau+\rho]\cap [0,T]$ of the curves in normal form are uniquely length-minimizing in the above sense.

\begin{subsection}{A normal form for regular abnormals}
We begin with the following result.

\begin{proposition}
\label{prop:nice}
    There exist $\varepsilon>0$, an open neighborhood $U$ of $q_0 =\gamma^{\lambda_0}(\tau)$, an open neighborhood $\Lambda$ of $\lambda_0$ in $\mathcal{R}_T$, and families of horizontal vector fields $Y^\lambda$ and $1$-forms $\omega^\lambda$ defined on $U$, depending smoothly on $\lambda \in \Lambda$, such that for all $\lambda \in \Lambda$ the following holds:
    \begin{itemize}
        \item [(1)] $\gamma^\lambda([ \tau-\varepsilon, \tau+\varepsilon]) \subset U$;
        \item[(2)] $g(Y^\lambda,Y^\lambda) \equiv 1$ on $U$;
        \item[(3)] $\gamma^\lambda |_{[\tau-\varepsilon,\tau+\varepsilon]}$ is an integral line of $Y^\lambda$;
        \item[(4)] $\omega^\lambda$ annihilates $\Delta$ along $\gamma^\lambda|_{[\tau-\varepsilon, \tau+\varepsilon]}$, $\omega^\lambda(q) \notin \Delta_3^\perp(q)$ for all $q \in U$ and the Lie derivative $\mathscr{L}_{Y^\lambda} \omega^\lambda$ vanishes identically along $\gamma^\lambda|_{[\tau-\varepsilon, \tau+\varepsilon]}$, 
    \end{itemize}
    with the understanding that $\varepsilon<\min\{\tau,T-\tau\}$ if $\tau \in (0,T)$, while the interval $[\tau-\varepsilon,\tau+\varepsilon]$ is replaced with $[0,\varepsilon]$ (resp.\ $[T-\varepsilon,T]$) if $\tau=0$ (resp.\ $\tau=T$). 
\end{proposition}

\begin{proof}
    To simplify the notation, we will prove the result for $\tau \in (0,T)$. This is not restrictive since we can always find an open neighborhood $\Lambda$ of $\lambda_0$ and a positive time $\delta>0$ such that the curves $(\gamma^\lambda)_{\lambda \in \Lambda}$ are well-defined on the interval $[-\delta, T+\delta]$.
    
    Fix a coordinate chart $(U_0,\psi)$ centered at $q_0$ and let $0<\varepsilon_0<\min\{\tau,T-\tau\}$ and an open neighborhood $\Lambda$ of $\lambda_0$ be such that 
    \begin{equation}
        \gamma^\lambda([\tau-\varepsilon_0,\tau+\varepsilon_0]) \subset U_0 \qquad \forall \lambda\in\Lambda. 
    \end{equation}
    Let us denote by $C(\varepsilon_0)$ the cube $(-\varepsilon_0,\varepsilon_0)^n$. Identifying $U_0$ with $\psi(U_0)$ we can assume that $\dot{\gamma}_2^{\lambda_0}(\tau)>0$ and, up to restricting $\Lambda$, also $\dot{\gamma}^\lambda_2(\tau)>0$ for all $\lambda \in \Lambda$. Consider the map $\Phi: \Lambda \times C(\varepsilon_0) \to \Lambda \times U_0$ that, denoting $x=(x_1,\dots,x_n)$, is defined as
    \begin{equation}
        \Phi(\lambda,x) = \left(\lambda, \gamma^\lambda(\tau+x_2) +(x_1,0,x_3,\dots,x_n)\right).
    \end{equation}
    Up to restricting $\varepsilon_0$ and $\Lambda$, such map is well-defined and its differential at $(\lambda_0,0)$ is invertible. By the inverse function theorem, we can find an open neighborhood $\mathcal{V}$ of $(\lambda_0,0)$ in $\Lambda \times C(\varepsilon_0)$ and an open neighborhood $\mathcal{W}$ of $(\lambda_0,q_0)$ in $\Lambda \times U_0$ such that $\Phi:\mathcal{V} \to \mathcal{W}$ is a diffeomorphism. By restricting $\Lambda$ if necessary, we can choose an open neighborhood $U_1 \subset U_0$ of $q_0$ such that $\Lambda \times U_1 \subset \mathcal{W}$, $\varepsilon_1>0$ such that 
    \begin{equation}
        \gamma^\lambda([\tau-\varepsilon_1,\tau+\varepsilon_1]) \subset U_1 \quad \forall \lambda \in \Lambda, 
    \end{equation} and replace $\mathcal{W}$ with $\Lambda \times U_1$. In particular, there exists a family of smooth functions $\theta^\lambda$ on $U_1$, depending smoothly on $\lambda \in \Lambda$, such that 
    \begin{equation}
        \Phi^{-1}(\lambda,q) = (\lambda, \theta^\lambda(q)), \quad \forall (\lambda,q) \in \Lambda \times U_1.
    \end{equation}
    For all $\lambda \in \Lambda$, the function $\theta^\lambda : U_1\to \R^n$ is a diffeomorphism from $U_1$ onto its image, and it defines a local chart on $U_1$ that rectifies the curve $\gamma^\lambda$, i.e:
    \begin{equation}
    \label{eq:rectif}
        \theta^\lambda(\gamma^\lambda(\tau+t)) = (0,t,0,\dots,0) 
    \end{equation}
    for all $\lambda \in \Lambda$ and $t \in [-\varepsilon_0, \varepsilon_0]$ such that $\gamma^\lambda(\tau+t) \in U_1$. Up to further restricting $\varepsilon_1$ and $\Lambda$, we can assume that $[-\varepsilon_1,\varepsilon_1]^n=:\bar{C}(\varepsilon_1) \subset \theta^\lambda(U_1) \subset C(\varepsilon_0)$ for all $\lambda \in \Lambda$. See Figure~\ref{fig:charts}.

    \begin{figure}[H]
     \centering
          \includegraphics[width=\textwidth]{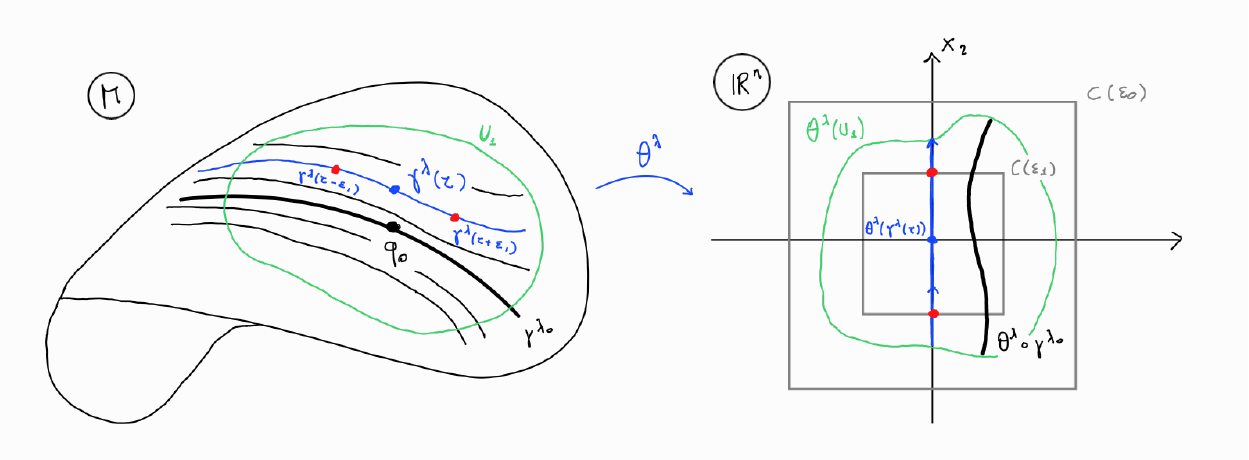}
     \caption{ Illustration of the maps $\theta^\lambda$. }\label{fig:charts}
\end{figure}

    Recall that $\varepsilon_0 < \min\{\tau,T-\tau\}$, and define the open set
    \[
    \Sigma:= \{(\lambda,t)\in \Lambda\times (-\varepsilon_0,\varepsilon_0) \mid \gamma^\lambda(\tau+t)\in U_1\}.
    \]
   Note that $\Lambda \times [-\varepsilon_1,\varepsilon_1] \subset \Sigma \subset \Lambda \times (-\varepsilon_0, \varepsilon_0)$. Since the curves  $(\gamma^\lambda)_{\lambda \in \Lambda}$ are horizontal and have unit speed, there exists a smooth function $\alpha:\Sigma \to \R^2$ such that
    \begin{equation}
        \dot\gamma^\lambda(\tau+t) = \sum_{i=1}^2 \alpha_i(\lambda,t)X_i(\gamma^\lambda(\tau+t))
    \end{equation}
    and $\alpha_1^2(\lambda,t)+\alpha_2^2(\lambda,t)=1$ for all $(\lambda,t) \in \Sigma$. Let
     $i:\Lambda \times \R \hookrightarrow \Lambda \times \R^n$  be the embedding $i(\lambda, x_2) = (\lambda, 0,x_2,0,\dots,0)$. By construction of $\Sigma$,
      we have that $(\Phi\circ i)(\Sigma)$ is a (properly) embedded submanifold of $\Lambda \times U_1$. In particular, there exists a smooth  $\hat{\alpha} :  \Lambda \times U_1  \to \R^2$ that extends $\alpha$ in the sense that  $\hat\alpha \circ \Phi \circ  i = \alpha$ on $\Sigma$. Since $\hat{\alpha}(\lambda_0,q_0) = \alpha(\lambda_0,0)\neq 0$, up to restricting $\Lambda$, $U_1$, and $\varepsilon_1$ we can assume that $\hat{\alpha}$ does not vanish on $\Lambda \times U_1$. Let then
    \begin{equation}
        a_i^\lambda(q) := \frac{\hat{\alpha}_i(\lambda,q)}{\sqrt{\hat{\alpha}_1^2(\lambda,q)+\hat{\alpha}_2^2(\lambda,q)}} \quad \forall (\lambda,q) \in \Lambda \times U_1, \ i=1,2,
    \end{equation}
    and 
    \begin{equation} 
        Y^\lambda(q):=\sum_{i=1}^2 a_i^\lambda(q) X_i(q),\qquad \forall (\lambda,q)\in \Lambda \times U_1.
    \end{equation}
    
    By construction, $(Y^\lambda)_{\lambda\in \Lambda}$ is a family of horizontal vector fields, smoothly depending on $\lambda$, with unit norm, such that $Y^\lambda(\gamma^\lambda(\tau+t)) = \dot\gamma^\lambda(\tau+t)$ for all $\lambda \in \Lambda$ and for all $t$ such that $\gamma^\lambda(\tau+t)\in U_1$, and so in particular for all $t\in [-\varepsilon_1,\varepsilon_1]$.

    Let us now turn to the construction of the $1$-forms $\omega^\lambda$. First of all we define  a family $(W^\lambda)_{\lambda \in \Lambda}$ of smooth vector fields on $\pi^{-1}(U_1)$ as 
    \begin{equation}
        W^\lambda(\mu) =  \sum_{i=1}^2a_i^\lambda(\pi(\mu)) \vec{h}_i(\mu) \quad \forall \mu \in \pi^{-1}(U_1), \lambda \in \Lambda.
    \end{equation}
    Notice that, by construction, $ e^{t W^\lambda} \circ e^{\tau \vec{G}}(\lambda)= e^{(\tau+t)\vec{G}}(\lambda)$ for all $(\lambda,t) \in \Sigma$. Let $S$ be the submanifold of $\R^n$ defined as $S=\{x_2=0\} \cap C(\varepsilon_1)$. Let us write in coordinates $\theta^\lambda$, 
    \begin{equation}
        ((\theta^\lambda)^{-1})^*e^{\tau\vec{G}}(\lambda)= \sum_{i=1}^n p_i^\lambda dx_i(0),
    \end{equation}
    and define the smooth sections of $T^*\R^n$ on $S$ letting for all $x=(x_1,0,x_3,\dots,x_n) \in S$
    \begin{equation}
        \beta^\lambda(x) = \sum_{i=1}^np_i^\lambda dx_i(x).
    \end{equation}
    By construction $(\theta^\lambda)^*(\beta^\lambda(0))=e^{\tau\vec{G}}(\lambda)$ for all $\lambda \in \Lambda$. We rewrite the map $\lambda\mapsto \beta^\lambda$ as a smooth form on $\Lambda \times S$ by setting
    \begin{equation}
        \beta(\lambda,(x_1,0,x_3,\dots,x_n)) := \beta^\lambda(x_1,0,x_3,\dots,x_n).
    \end{equation}
    Since $\Phi$ is a diffeomorphism, the image $\Phi(\Lambda \times S)$ is an embedded submanifold of $\Lambda \times M$ of codimension $1$. We equip $\Phi(\Lambda \times S)$ with the smooth form $\eta := (\Phi^{-1})^*\beta$. By construction, $\eta(\lambda,\gamma^\lambda(\tau))=e^{\tau\vec{G}}(\lambda)$ for all $\lambda \in \Lambda$. Let now
    \begin{equation}
        A:(-\varepsilon_1,\varepsilon_1) \times \Phi(\Lambda \times S) \to \Lambda \times M
    \end{equation}
    be defined by 
    \begin{equation}
        A(s,(\lambda,q))=\left( \lambda, \pi \circ e^{s W^\lambda}(\eta(\lambda,q)) \right).
    \end{equation}
    Up to reducing $\Lambda$ and $\varepsilon_1$, the map $A$ is well-defined, moreover $A(0,\cdot)=\mathrm{Id}_{\Phi(\Lambda\times S)}$ and 
    \begin{equation}
    \left.    \frac{d}{ds}\right |_{s=0} A(s,(\lambda_0,q_0))=  \Big(0,\pi_* \big (W^{\lambda_0}(e^{\tau\vec{G}}(\lambda_0))\big)\Big) = \big(0,\dot{\gamma}^{\lambda_0}(\tau)\big).
    \end{equation}
    Hence, the differential of $A$ is invertible at $(0,(\lambda_0,q_0))$, since $S$ is transverse to the support of $\gamma^{\lambda_0}$ at $\gamma^{\lambda_0}(\tau)$ in coordinates $\theta^\lambda$. Again, there exist open neighborhoods $\mathcal{V}'$ of $(0,(\lambda_0,q_0))$ and $\mathcal{W}'$ of $(\lambda_0,q_0)$ such that $A:\mathcal{V}' \to \mathcal{W}'$ is a diffeomorphism. We can now choose an open neighborhood $U_2 \subset U_1$ of $q_0$ and $0<\varepsilon_2\le \varepsilon_1$ such that 
    \begin{equation}
        \gamma^\lambda([\tau-\varepsilon_2,\tau+\varepsilon_2]) \subset U_2 \quad \forall \lambda \in \Lambda, 
    \end{equation} 
    possibly reducing $\Lambda$, and assume that $\mathcal{W}'=\Lambda \times U_2$. For all $\lambda \in \Lambda$ we define the $1$-form $\omega^\lambda$ on $U_2$ as follows. For all $q \in U_2$ there exist unique $s \in (-\varepsilon_1,\varepsilon_1)$ and, by definition of $\Phi$, $\tilde{q} \in (\theta^\lambda)^{-1}(S)$ such that $(\lambda,q)=A(s,(\lambda,\tilde{q}))$, therefore we set
    \begin{equation}
        \omega^\lambda(q):=e^{s W^\lambda}(\eta(\lambda,\tilde{q})).
    \end{equation}
    By construction, the family of $1$-forms $\omega^\lambda$ depends smoothly on $\lambda$, each $\omega^\lambda$ annihilates $\Delta$ (actually, $\Delta_2$) along $\gamma^\lambda|_{[\tau-\varepsilon_2,\tau+\varepsilon_2]}$ and, up to further restricting $\Lambda$, $\varepsilon_2$, and $U_2$, it never intersect $\Delta_3^\perp$. Notice that for all $\lambda \in \Lambda$ and $t \in [-\varepsilon_2,\varepsilon_2]$ it holds $(\lambda,\gamma^\lambda(\tau+t))= A(t,(\lambda, \gamma^\lambda(\tau)))$, hence
    \begin{equation}
    \label{eq:omegaflow}
        \omega^\lambda(\gamma^\lambda(\tau+t))= e^{t W^\lambda}(e^{\tau\vec{G}}(\lambda))=e^{(\tau+t)\vec{G}}(\lambda)
    \end{equation}
    for all $t \in [-\varepsilon_2,\varepsilon_2]$.   We can now fix $U:=U_2$ and $\varepsilon:=\varepsilon_2$ as in the statement, since there will be no more restrictions.
    \\

    The proof that $\mathscr{L}_{Y^\lambda} \omega^\lambda \equiv 0$ along $\gamma^\lambda|_{[\tau-\varepsilon,\tau+\varepsilon]}$ is given for fixed $\lambda$, and follows without modifications the one given in \cite{liu-sus}. We report it for the sake of completeness. \\
    
    For $t \in [\tau-\varepsilon,\tau+\varepsilon]$ let $\Gamma^\lambda(t):= e^{t\vec{G}}(\lambda)$. First observe that, by construction of $W^\lambda$, it holds $W^\lambda=\vec{h}_{Y^\lambda}$ along $\Gamma^\lambda$. Indeed, $Y^\lambda=a_1^\lambda X_1 + a_2^\lambda X_2$ hence
    \begin{align}
        \vec{h}_{Y^\lambda}& = a_1^\lambda \ \vec{h}_1 + a_2^\lambda \ \vec{h}_2 + \vec{a}_1^\lambda \ h_1 + \vec{a}_2^\lambda \ h_2 \\ 
         & = W^\lambda + \vec{a}_1^\lambda \ h_1 + \vec{a}_2^\lambda \ h_2,
    \end{align}
    while $h_1=h_2=0$ along $\Gamma^\lambda$. This, together with \eqref{eq:omegaflow}, implies that
    \begin{equation}
        \frac{d}{d t}\Gamma^\lambda(t)=\vec{h}_{Y^\lambda}(\Gamma^\lambda(t)).
    \end{equation}
    Let now $Z \in \Vect(U)$.
    From the previous equality it follows that
    \begin{equation}
        \frac{d}{d t}h_Z(\Gamma^\lambda(t))= h_{[Y^\lambda,Z]}(\Gamma^\lambda(t))=(\scal{\omega^\lambda}{[Y^\lambda,Z]})(\gamma^\lambda(t)). 
    \end{equation}
    On the other hand, 
    \begin{equation}
        \frac{d}{d t}h_Z(\Gamma^\lambda(t)) = \frac{d}{d t}\scal{\omega^\lambda}{Z}(\gamma^\lambda(t)) = (Y^\lambda \scal{\omega^\lambda}{Z})(\gamma^\lambda(t)).
    \end{equation} 
    Recalling the identity
    \begin{equation}
        \scal{\mathscr{L}_{Y^\lambda}\omega^\lambda}{Z}=Y^\lambda \scal{\omega^\lambda}{Z} - \scal{\omega^\lambda}{[Y^\lambda,Z]},
    \end{equation}
    we obtain that $\mathscr{L}_{Y^\lambda} \omega^\lambda \equiv 0$ along $\gamma^\lambda|_{[\tau-\varepsilon,\tau+\varepsilon]}$.
\end{proof}

We now prove that short arcs of regular abnormal paths taken as in Proposition~\ref{prop:nice} can be put simultaneously in a normal form via a family of coordinate charts parametrized smoothly by $\lambda \in \Lambda$. The proof is based on two results that we introduce now. 

\begin{definition}
    Let $M$ be a smooth manifold, $\mathcal{D}$ a smooth distribution on $M$, $X$ a smooth section of $\mathcal{D}$, and $\gamma:[0,T] \to M$ a smooth curve on $M$. Then $\mathcal{D}$ is \emph{invariant under $X$ along $\gamma$} if for every smooth section $Y$ of $\mathcal{D}$ and for all $t\in [0,T]$ it holds $[X,Y](\gamma(t)) \in \mathcal{D}(\gamma(t))$.
\end{definition}

\begin{remark}
    The above definition is weaker than the usual notion of invariance of a distribution under a vector field, since it is required 
    only along a reference curve.
\end{remark}

\begin{lemma}
\label{lem:Lemma3}
In the setting of Proposition \ref{prop:nice}, let
      \begin{equation}
          \mathcal{D^\lambda}=\spn\{X, Y^\lambda,X_4^\lambda, \dots, X_n^\lambda\}
      \end{equation} 
      be a family of smooth distributions of rank $n-1$ on $U$, where the vector fields $(X_i^\lambda)_{i=4,\dots,n}$ depend smoothly on $\lambda \in \Lambda$. Assume that for all $\lambda \in \Lambda$ the distribution $\mathcal{D}^\lambda$ is invariant under $Y^\lambda$ along $\gamma^\lambda|_{[\tau-\varepsilon,\tau+\varepsilon] \cap [0,T]}$. Then, up to restricting $\Lambda$, $U$, and $\varepsilon$ there exists a family of charts $\kappa^\lambda=(x_1,\dots,x_n)$ defined on $U$, each centered at $\gamma^\lambda(\tau)$ and whose image contains a closed cube $\bar{C}(\varepsilon)$, with respect to which on $\bar{C}(\varepsilon)$ it holds
      \begin{equation}
      \label{eq:XY}
          X = \partial_1, \quad Y^\lambda = \partial_2 + x_1 \sum_{i=1}^n \psi_i^\lambda(x)\partial_i,
      \end{equation}
      where $\psi_i^\lambda  : \bar{C}(\varepsilon)\to \R$ are smooth and depend smoothly on $\lambda \in \Lambda$, and
      \begin{equation}
      \label{eq:psi3}
          \psi_3^\lambda(x) = x_1 \eta_1^\lambda(x) +  x_3 \eta_3^\lambda(x) + \dots +  x_n \eta_n^\lambda(x)
      \end{equation}
      for some smooth functions  $\eta_i^\lambda : \bar{C}(\varepsilon)\to \R$, for $i\neq 2$, depending smoothly on $\lambda \in \Lambda$. Moreover, if it holds $[X,[X,Y^{\lambda_0}]](q_0)\notin \mathcal{D}^\lambda(q_0)$, then $\kappa^\lambda$  can be defined in such a way that $\eta_1^\lambda \neq 0$ on $\bar{C}(\varepsilon)$ for all $\lambda \in \Lambda$. \\
\end{lemma}

The proof of this lemma needs the following observations, that we state without proof after introducing some notation.

\begin{notation}
    Let $M,N$ be smooth manifolds, $q \in M$, and $\Phi: M \to N$ a smooth diffeomorphism. If $v \in T_qM$ we denote by
    \begin{equation}
        \Phi_* v \in T_{\Phi(q)}N
    \end{equation}
    the differential of $\Phi$ at $q$ applied to $v$, namely $\Phi_* v = d\Phi(q) v$. Moreover, if $X \in \Vect(M)$ we let $\Phi_*X$ be the push-forward of $X$ through $\Phi$. With this notation,
    \begin{equation}
        (\Phi_*X)(\Phi(q))= \Phi_*(X(q)).
    \end{equation}
\end{notation}

\begin{lemma}[{\cite[Lemma 4]{liu-sus}}]
    \label{lem:Lemma4}
    Let $M$ be a smooth manifold, $X \in \Vect(M)$, $\mathcal{D}$ a smooth distribution, and $q \in M$. Let $\gamma(t) = e^{t X}(q)$ and assume that $\mathcal{D}$ is invariant under $X$ along $\gamma$. Then for each times $t_1,t_2$ at which $\gamma$ is defined, letting $q_i= e^{t_i X}(q)$, $i=1,2$, it holds
    \begin{equation}
        \left(e^{(t_1-t_2)X}\right)_* v \in \mathcal{D}(q_1) \qquad \forall v \in \mathcal{D}(q_2).
    \end{equation}
\end{lemma}

\begin{lemma}[{\cite[Chapter 2]{ABB}}]
\label{lem:Phi}
    Let $M$ be a smooth manifold and $X_1,\dots,X_n \in \Vect(M)$ be linearly independent at $q \in M$. The map
    \begin{equation}
        \Phi: \R^n \to M , \quad \Phi(x_1,\dots,x_n)= e^{x_1X_1}\circ \dots \circ e^{x_nX_n}(q)
    \end{equation}
    is a local diffeomorphism in a neighborhood of $0$ and, letting $x=(x_1,\dots,x_n)$, we have
    \begin{equation}
        \Phi_* \left(\partial_{i}(x)\right) = \left( (e^{x_1X_1})_* \circ \dots \circ (e^{x_{i}X_{i}})_* X_i\right)(\Phi(x)).
    \end{equation}
\end{lemma}

\begin{proof}[Proof of Lemma \ref{lem:Lemma3}]
     Again we can assume that $\tau \in (0,T)$. Let $X_1:=X$ and $X_2^\lambda:=Y^\lambda$ for all $\lambda \in \Lambda$. Up to restricting $\Lambda$, $U$, and $\varepsilon$, we find an additional family of vector fields $X_3^\lambda$, depending smoothly on $\lambda$ and defined on $U$, such that $X_1,X_2^\lambda,X_3^\lambda,\dots,X_n^\lambda$ form a local frame  for all $\lambda \in \Lambda$.

    Let $r>0$ and consider the map $\Psi:\Lambda \times C(r) \to \Lambda \times U$ defined as
    \begin{equation}
        \Psi(\lambda,x)=\left( \lambda, e^{x_1X_1}\circ e^{x_2X_2^\lambda}\circ e^{x_3X_3^\lambda}\circ \cdots \circ e^{x_nX_n^\lambda}(\gamma^\lambda(\tau))\right),
    \end{equation}
    which is well-defined (up to reducing $\Lambda$ and $r$) and smooth. The differential of $\Psi$ at $(\lambda_0,0)$ is invertible, hence by the inverse function theorem we can find an open neighborhood $\mathcal{V}$ of $(\lambda_0,0)$ and an open neighborhood $\mathcal{W}$ of $(\lambda_0,q_0)$ such that $\Psi:\mathcal{V} \to \mathcal{W}$ is a diffeomorphism. Possibly restricting $\Lambda$, $U$, and $\varepsilon$, we assume that $\mathcal{W}= \Lambda \times U$. The inverse of $\Psi$ on $\mathcal{W}$ verifies
    \begin{equation}
        \Psi^{-1}(\lambda,q) = (\lambda, \kappa^\lambda(q)) \quad \forall (\lambda,q) \in \Lambda \times U,
    \end{equation}
    where the maps $\kappa^\lambda$ are a smooth family of diffeomorphisms from $U$ onto their images. Such maps define smooth charts on $U$, centered at $\gamma^\lambda(\tau)$, that rectify the curves  $\gamma^\lambda|_{[\tau-\varepsilon,\tau+\varepsilon]}$. More precisely, since $\gamma^\lambda|_{[\tau-\varepsilon,\tau+\varepsilon]}$ are integral curves of $X^\lambda_2$, we have that
    \begin{equation}
        \kappa^\lambda(\gamma^\lambda(\tau+t)) = (0,t,0,\dots,0),\qquad \forall t\in [-\varepsilon,\varepsilon],\, \lambda \in \Lambda.
    \end{equation}
     Up to further restricting $\Lambda$ and $\varepsilon$ we have $\bar{C}(\varepsilon) \subset \kappa^\lambda(U) \subset C(r)$ for all $\lambda \in \Lambda$. \\
 
    Let us compute $X_1$ and $X_2^\lambda$ in coordinates $\kappa^\lambda$. First of all we have $X_1=\partial_1$. Moreover $X_2^\lambda(x) = \partial_2(x)$ 
    at each point $x=(x_1,\dots,x_n)$ with $x_1=0$, hence 
    \begin{equation}
        X_2^\lambda = \partial_2 + x_1 \sum_{i=1}^n \psi_i^\lambda \partial_i,
    \end{equation}
     with the functions $\psi_i^\lambda$ depending smoothly on $\lambda$. We prove that \eqref{eq:psi3} holds. It is enough to show that $\psi_3^\lambda=0$ on $\{(0,x_2,0\dots,0) \mid x_2 \in [-\varepsilon,\varepsilon]\}$, that is, along the support of $\gamma^\lambda|_{[\tau-\varepsilon,\tau+\varepsilon]}$ in coordinates $\kappa^\lambda$. Notice that
    \begin{equation}\label{eq:bracket}
        [X_1,X_2^\lambda] = \sum_{i=1}^n \left (\psi_i^\lambda + x_1 \frac{\partial \psi_i^\lambda}{\partial x_1} \right) \partial_i.
    \end{equation}
    Thus, we just need to prove that $[X_1,X_2^\lambda]$ is a linear combination of the $(\partial_i)_{i\neq 3}$ along $\gamma^\lambda|_{[\tau-\varepsilon,\tau+\varepsilon]}$. To prove it, recall that $\partial_1 = X_1$ and, by Lemma \ref{lem:Phi}, we have
    \begin{equation}\label{eq:todo}
        \partial_i(0,x_2,0,\dots,0)= \left(e^{x_2X_2^\lambda}\right)_*X_i^\lambda \quad \forall i\geq 2, \, x_2 \in [-\varepsilon,\varepsilon].
    \end{equation}
    Using these expressions in \eqref{eq:bracket}, and noting that $[X_1,X_2^\lambda](\gamma^\lambda(\tau+t)) \in \mathcal{D}^\lambda(\gamma^\lambda(\tau+t))$ for all $|t|<\varepsilon$ (by the invariance of $\mathcal{D}^\lambda$ with respect to $X_2^\lambda$ along $\gamma^\lambda$) we deduce that, on $\gamma^\lambda$, there is no $\partial_3$ component in \eqref{eq:bracket}.
      
    To conclude, note that if $[X,[X,Y^{\lambda_0}]](q_0) \notin \mathcal{D}^\lambda(q_0)$, up to further shrinking $\Lambda, U_1$, and $\varepsilon$, we can assume that the same holds for $Y^\lambda$ at each $\gamma^\lambda(\tau)$, and let $X_3^\lambda=[X,[X,Y^\lambda]]$. With this choice we find that, again in coordinates $\kappa^\lambda$,
    \begin{equation}
        X_3^\lambda = [X_1,[X_1,X_2^\lambda]]= \sum_{i=1}^n \left( 2\frac{\partial \psi_i^\lambda}{\partial x_1} + x_1 \frac{\partial^2\psi_i^\lambda}{\partial x_1^2} \right) \partial_i,
    \end{equation}
    and since $X_3(\gamma^\lambda(\tau)) = \partial_3(0)$ we deduce that $\eta_1^\lambda(0) \neq 0$ for all $\lambda \in \Lambda$. If necessary, we can restrict $\Lambda$ and $\varepsilon$ so that $\eta_1^\lambda$ never vanishes on $\bar{C}(\varepsilon)$ for all $\lambda \in \Lambda$.
\end{proof}

\begin{proposition}
\label{prop:thm4}
    In the same setting as Proposition \ref{prop:nice}, there exist an open neighborhood $\Lambda$ of $\lambda_0$, an open neighborhood $U$ of $q_0$,  a family of coordinate charts $(\kappa^\lambda)_{\lambda \in \Lambda}$ defined on $U$, depending smoothly on $\lambda$, and $\varepsilon>0$ such that
    \begin{itemize}
        \item  the curves $\gamma^\lambda |_{[\tau-\varepsilon, \tau+\varepsilon]}$ are contained in $U$ for all $\lambda \in \Lambda$;
        \item  each $\kappa^\lambda$ is centered at $\gamma^\lambda(\tau)$ and verifies $\bar{C}(\varepsilon) \subset \kappa^\lambda(U)$;
        \item there exist $X \in \Vect(U)$, $(Y^\lambda)_{\lambda \in \Lambda} \subset \Vect(U)$ depending smoothly on $\lambda$, such that for all $\lambda \in \Lambda$ it holds $\Delta = \spn\{X,Y^\lambda\}$ on $U$;
    \end{itemize} 
    so that for all $\lambda \in \Lambda$ the following holds:
    \begin{itemize}
        \item [(i)] in the chart $\kappa^\lambda$, on the cube $\bar{C}(\varepsilon)$ the vector fields $X$ and $Y^\lambda$ are written  as \eqref{eq:XY} and $\psi_3^\lambda$ verifies \eqref{eq:psi3} with $\eta_1^\lambda$ never vanishing;
        \item [(ii)]in the chart $\kappa^\lambda$ it holds $\gamma^\lambda(\tau+t)=(0,t,0,\dots,0)$ for all $t \in [-\varepsilon,\varepsilon]$;
    \end{itemize}
 with the understanding that $\varepsilon<\min\{\tau,T-\tau\}$ if $\tau \in (0,T)$, while the interval $[\tau-\varepsilon,\tau+\varepsilon]$ is replaced with $[0,\varepsilon]$ (resp.\ $[T-\varepsilon,T]$) if $\tau=0$ (resp.\ $\tau=T$). 
\end{proposition}

\begin{proof}
    Again we can assume that $\tau \in (0,T)$. Let $\varepsilon, \Lambda, U, Y^\lambda, \omega^\lambda$ be as in Proposition \ref{prop:nice}, so that \emph{(1),(2),(3),(4)} are verified. Up to further shrinking $\Lambda$, $U$, and $\varepsilon$, we can choose $X\in \Vect(U)$ so that $\Delta=\mathrm{span}\{X,Y^\lambda\}$ on $U$ for all $\lambda \in \Lambda$. For the rest of the proof, ``along $\gamma^\lambda$'' is understood as ``along $\gamma^\lambda|_{[\tau-\varepsilon, \tau+\varepsilon]}$''.
     
    Consider the distributions of rank $n-1$ on $U$ defined as $\mathcal{D}^\lambda=\ker(\omega^\lambda)$. Since $\mathscr{L}_{Y^\lambda} \omega^\lambda \equiv 0$ along $\gamma^\lambda$, for any $Z \in \Vect(U)$ it holds
    \begin{equation}
        \scal{\omega^\lambda}{[Y^\lambda,Z]} = Y^\lambda(\scal{\omega^\lambda}{Z})- \scal{\mathscr{L}_{Y^\lambda} \omega^\lambda}{Z} = Y^\lambda(\scal{\omega^\lambda}{Z})
    \end{equation}
    along $\gamma^\lambda$. Now:
    \begin{itemize}
        \item taking a smooth section $Z$ of $\mathcal{D}^\lambda$, we deduce that $\mathcal{D}^\lambda$ is invariant under $Y^\lambda$ along $\gamma^\lambda$;
        \item taking $Z=X$ we find that $[X,Y^\lambda] \in \mathcal{D}^\lambda$ along $\gamma^\lambda$, since $\omega^\lambda$ annihilates $\Delta$ along $\gamma^\lambda$;
        \item taking $Z=[X,Y^\lambda]$ we deduce that $\scal{\omega^\lambda}{[Y^\lambda,[X,Y^\lambda]]}=0$ along $\gamma^\lambda$.
    \end{itemize}   
    Since $\omega^\lambda \notin \Delta_3^\perp$ on $U$ we conclude that $[X,[X,Y^\lambda]] \notin \mathcal{D}^\lambda$ along $\gamma^\lambda$. Up to restricting $\Lambda$, $U$, and $\varepsilon$, we can assume that $[X,[X,Y^\lambda]](q) \notin \mathcal{D}^\lambda(q)$ for all $q \in U$ and $\lambda \in \Lambda$.

    Notice that by construction $X,Y^\lambda \in \mathcal{D}^\lambda$ only along $\gamma^\lambda$. Finally, we modify $\mathcal{D}^\lambda$ into a distribution $\bar{\mathcal{D}}^\lambda$ so that $X,Y^\lambda$ are sections of $\bar{\mathcal{D}}^\lambda$: then, we can apply Lemma \ref{lem:Lemma3} and conclude. 
    
    To this aim, we construct a basis of sections $\{Z_1^\lambda,Z_2^\lambda, Z_4^\lambda,\dots,Z_n^\lambda\}$ of $\mathcal{D}^\lambda$ on $U$, chosen so that $Z_1^\lambda=X$ and $Z_2^\lambda=Y^\lambda$ along $\gamma^\lambda$ and $Z_i^\lambda$ depend smoothly on $\lambda$ (up to shrinking $\varepsilon$ and $U$ uniformly in $\lambda$). Let $\hat{g}$ be a Riemannian metric that extends $g$ on $U$ (possibly restricting $U$, and $\varepsilon$ as a consequence). 
    For every $q \in U$, let $Z_1^\lambda(q)$ and $Z_2^\lambda(q)$ be the $\hat{g}_q$-orthogonal projections of $X(q)$ and $Y^\lambda(q)$ respectively on $\mathcal{D}^\lambda(q)$. Notice that since $\omega^\lambda$ annihilates $\Delta$ along $\gamma^\lambda$ and $\mathcal{D}^\lambda=\ker(\omega^\lambda)$, then $Z_1^\lambda = X$ and $Z_2^\lambda=Y^\lambda$ along $\gamma^\lambda$. By shrinking $\Lambda$, $U$, and $\varepsilon$ if necessary, we can find $X_3,\dots,X_n \in \Vect(U)$ such that $Z_1^\lambda,Z_2^\lambda,X_3,\dots,X_n$ are linearly independent on $U$ and, for $i=3,\dots,n$, consider the pointwise $\hat{g}$-orthogonal projection $\hat{Z}_i^\lambda$ of $X_i$ on $\mathcal{D}^\lambda$. Up to shrinking $\Lambda$, $U$, and $\varepsilon$, we can find $j_1,\dots,j_{n-3} \in \{3,\dots,n\}$ such that $\mathcal{D}^\lambda= \spn\{Z_1^\lambda,Z_2^\lambda,\hat{Z}_{j_1}^\lambda, \dots , \hat{Z}_{j_{n-3}}^\lambda\}$ on $U$. Notice that we might need to shrink $\Lambda$ in order to select the \emph{same} indices $j_1,\dots,j_{n-3}$ for all the $\lambda \in \Lambda$.  We conclude by setting $Z_i^\lambda = \hat{Z}_{j_{i-3}}^\lambda$ for $i=4,\dots,n$.
    
    
   For each $\lambda \in \Lambda$, define the new distribution $\bar{\mathcal{D}}^\lambda$ on $U$ as
    \begin{equation}
        \bar{\mathcal{D}}^\lambda = \spn\{X,Y^\lambda, Z_4^\lambda,\dots, Z_n^\lambda\}. 
    \end{equation}
    Since $\bar{\mathcal{D}}^\lambda = \mathcal{D}^\lambda$ along $\gamma^\lambda$, we have $[X,[X,Y^\lambda]](\gamma^\lambda(\tau+t)) \notin \bar{\mathcal{D}}^\lambda(\gamma^\lambda(\tau+t))$ for all $t \in [-\varepsilon,\varepsilon]$. \\
    
    We prove the invariance of $\bar{\mathcal{D}}^\lambda$ under $Y^\lambda$ along $\gamma^\lambda$. If $Z$ is a smooth section of $\bar{\mathcal{D}}^\lambda$ we can rewrite
    \begin{equation}
        Z = \nu_1^\lambda X + \nu_2^\lambda Y^\lambda + \sum_{i=4}^{n} \nu_i^\lambda Z_i^\lambda,
    \end{equation}
    for some smooth functions $\nu_i^\lambda$ depending smoothly on $\lambda$. Then
    \begin{equation}
        [Y^\lambda,Z] = \nu_1^\lambda[Y^\lambda,X]+ \sum_{i=4}^{n}\nu_i^\lambda [Y^\lambda,Z_i^\lambda] 
         + (Y^\lambda \nu_1^\lambda) X + (Y^\lambda \nu_2^\lambda) Y^\lambda + \sum_{i=4}^{n} (Y^\lambda \nu_i^\lambda) Z_i^\lambda.
    \end{equation}
    By invariance of $\mathcal{D}^\lambda$ along $\gamma^\lambda$ we have that $[Y^\lambda,X]$ and $[Y^\lambda,Z_i^\lambda]$ are sections of $\mathcal{D}^\lambda$ along $\gamma^\lambda$, as well as $X,Y^\lambda, Z_i^\lambda$ for $i \ge 4$. Since $\bar{\mathcal{D}}^\lambda$ and $\mathcal{D}^\lambda$ coincide on $\gamma^\lambda$, we conclude.
\end{proof}
\begin{subsection}{Proof of Theorem \ref{thm:opt_cpt}}
Let $U,\Lambda,\kappa^\lambda,\varepsilon$ be as in Proposition \ref{prop:thm4}, again assuming that $\tau \in (0,T)$ without loss of generality. Fix $0<\varepsilon'< \varepsilon$. For each $\lambda \in \Lambda$ we replace $\gamma^\lambda$ with its restriction to $[\tau-\varepsilon', \tau+\varepsilon']$ so that, in coordinates $\kappa^\lambda$, we have
\begin{equation}
    \gamma^\lambda(\tau+t)=(0,t,0,\dots,0)=: x^\lambda(t) 
    , \quad \forall t\in[-\varepsilon',\varepsilon'].
\end{equation}
Notice that $x^\lambda :[-\varepsilon',\varepsilon']\to \bar{C}(\varepsilon)\subset \R^n$ is the trajectory that passes through $x=0$ at time $0$ corresponding to the control $(u,v) \equiv (0,1)$ of the control system in $\bar{C}(\varepsilon)$:
\begin{align}\label{eq:sigmalambda}
    \begin{cases}
        \dot{x}_1= u + v x_1 \psi_1^\lambda(x), \\
        \dot{x}_2= (1+x_1\psi_2^\lambda(x))v,\\
        \dot{x}_i= vx_1\psi_i^\lambda(x), \quad i=3,\dots,n,
    \end{cases} \quad \quad u^2+v^2 = 1, \tag{$\Sigma^\lambda$}
\end{align}
where
\begin{equation}
    \psi_3^\lambda(x)=x_1\eta_1^\lambda(x) + x_3\eta_3^\lambda(x) + \dots + x_n\eta_n^\lambda(x).
\end{equation}

The admissible trajectories $y:[a,b] \to \bar{C}(\varepsilon)$ for the system \eqref{eq:sigmalambda} with control $(u,v)$ are the expressions in coordinates $\kappa^\lambda$ of the horizontal curves $\gamma$ on $(\kappa^\lambda)^{-1}(\bar{C}(\varepsilon))$ that satisfy 
\begin{equation}
    \dot{\gamma}(t) = u(t) X(\gamma(t)) + v(t) Y^\lambda(\gamma(t)) \quad \text{for a.e. } t \in [a,b].
\end{equation}
Fix an open set $W \subset \R^n$ such that $\bar{C}(\varepsilon') \Subset W \Subset C(\varepsilon)$. Possibly restricting $\Lambda$ to ensure that the functions $\psi_i^\lambda$ and $\eta_i^\lambda$ are well defined also on the closure $\bar{\Lambda} \times \bar{C}(\varepsilon)$, let
\begin{align}
    C_1 &= \max \{ |\dot{y}(\cdot)| \mid y \ \text{is a trajectory of {\eqref{eq:sigmalambda}}}  \text{ contained in } \bar{W}, \ \lambda \in \bar{\Lambda} \} >0, \\
    C_2 &= \max \{ |\psi_i^\lambda(x)| \mid x \in \bar{W}, \ i=1,\dots,n, \ \lambda \in \bar{\Lambda}\}>0, \\
    C_3 &= \max \left \{|\eta_i^\lambda(x)| \mid x\in \bar{W}, \ i=3,\dots,n, \ \lambda \in \bar{\Lambda}\right \}>0, \\ 
    m &= \min \left \{ |\eta_1^\lambda(x^\lambda (t))| \, \Big | -\varepsilon'\le t \le \varepsilon', \ \lambda \in \bar{\Lambda} \right \}>0. 
\end{align}

Fix moreover $\delta= \operatorname{dist}(C(\varepsilon'),W^c)>0$ and $L>0$ such that all the $(\eta_i^\lambda)_{ \lambda \in \Lambda}$ are $L$-Lipschitz continuous on $W$. Then set
\begin{equation}
    \rho_1 = \frac{\delta}{4C_1}, \quad 
    \rho_2 = \frac{m}{8C_1L}, \quad
    \rho_3 = \frac{m}{8(n-2)C_2C_3}, \quad
    \rho_4 = \frac{1}{7C_1C_2}.
\end{equation}

\begin{remark}
\label{rk:tau12}
    With these choices:
    \begin{itemize}
        \item if $y:[t_1,t_2] \to C(\varepsilon)$ is a trajectory of \eqref{eq:sigmalambda} through a point in $\bar{C}(\varepsilon')$ with $t_2-t_1 \le 2\rho_1$ then it must be contained in $W$;
        \item if $x,y \in W$ verify $|x-y| \le 2C_1 \rho_2$ then $|\eta_1^\lambda(x) - \eta_1^\lambda(y)| \le m/4$.
    \end{itemize}   
\end{remark}

Let $\rho = \min \{\varepsilon', \rho_1,\rho_2,\rho_3,\rho_4\}$ and from now on for all $\lambda \in \Lambda$ replace $\gamma^\lambda$ with its restriction to the interval $[\tau-\rho,\tau+\rho]$. Since $\gamma^\lambda$ has unit speed, we have $L(\gamma^\lambda)=2\rho$. The proof is concluded if we prove that if $\bar{\gamma}:[s_1,s_2]\to M$ is a length-minimizing trajectory parametrized with unit speed with the same endpoints as $\gamma^\lambda$ for some $\lambda \in \Lambda$, then $L(\bar{\gamma})=s_2-s_1 = 2\rho$ and $\bar{\gamma}(s)=\gamma^\lambda((\tau-\rho)+s-s_1)$ for all $s \in [s_1,s_2]$. 

 Notice that any such $\bar{\gamma}$  must be contained in $(\kappa^\lambda)^{-1}(C(\varepsilon))$. Otherwise, there would exist $s^* \in (s_1,s_2)$ such that $s^* = \sup \{s\in [s_1,s_2] \mid \kappa^\lambda \circ \bar{\gamma} \,([0,s]) \in C(\varepsilon)\}$. Since $\bar{\gamma}$ is a unit speed length-minimizer, then $s_2-s_1 = L(\bar{\gamma} \,)\le 2 \rho$. But then $\kappa^\lambda \circ \bar{\gamma}|_{[s_1,s^*]}$ would be a trajectory of \eqref{eq:sigmalambda} through a point in $\bar{C}(\varepsilon')$ with $s^*-s_1 \le s_2-s_1 \le 2\rho_1$, and by Remark~\ref{rk:tau12} it would be contained in $W$, contradicting the maximality of $s^*$. It follows also that such $\bar\gamma$, written in coordinates $\kappa^\lambda$, must be an admissible trajectory of \eqref{eq:sigmalambda}.\\

The rest of the proof is given for fixed $\lambda$ and follows with small modifications the one given in \cite{liu-sus}. We report it for the sake of completeness. Assume that $y :[s_1,s_2] \to \R^n$ is another trajectory of \eqref{eq:sigmalambda}, with control $(u_y,v_y)$, that joins $x^\lambda(-\rho)$ to $x^\lambda(\rho)$ in time $\sigma = s_2 - s_1$.  Assume that $\sigma \le 2\rho$. 

Denote $y(s)=(y_1(s),\dots,y_n(s))$ and $h(s) = \int_{s_1}^s v_y(r)\,dr$. Then $|h(s)| \le s-s_1$ for all $s\in [s_1,s_2]$, so $-\sigma \le h(s_2) \le \sigma$. Let 
    \begin{equation}
       0\leq  \alpha = \sigma - h(s_2), \quad \beta = \max\{|y_1(s)| \mid s_1 \le s \le s_2\}.
    \end{equation} 
    We then have
    \begin{align}
    \label{eq:2rho}
        2\rho & = y_2(s_2) - y_2(s_1)
        = \int_{s_1}^{s_2} \left(1 + y_1(s)\psi_2^\lambda(y(s))\right) v_y(s) \ ds \\
        & \le \int_{s_1}^{s_2} v_y(s)\,ds + \beta\sigma C_2
        = h(s_2) + \beta\sigma C_2
        = \sigma - \alpha + \beta\sigma C_2.
    \end{align}

We now claim that
\begin{equation}\label{eq:upper_bound_beta}
    \beta \le 3C_1\alpha.
\end{equation}
Let us write
\begin{equation}
\int_{s_1}^{s_2} y_1(s)^2\,ds = \int_{s_1}^{s_2} y_1(s)^2(1-v_y(s))\,ds
+ \int_{s_1}^{s_2} y_1(s)^2 v_y(s)\,ds = \mathcal{I}_1 + \mathcal{I}_2.
\end{equation}
Notice that
$\mathcal{I}_1 \le \beta^2\alpha$.
The second integral is 
\begin{align}
\mathcal{I}_2 &= \int_{s_1}^{s_2} y_1(s)^2 v_y(s) \frac{\eta_1^\lambda(y(s_1))}{\eta_1^\lambda(y(s_1))}\,ds \\ 
&= \int_{s_1}^{s_2} y_1(s)^2 v_y(s) \frac{\eta_1^\lambda(y(s_1))-\eta_1^\lambda(y(s))} {\eta_1^\lambda(y(s_1))}\,ds+ 
\int_{s_1}^{s_2} y_1(s)^2 v_y(s) \frac{\eta_1^\lambda(y(s))} {\eta_1^\lambda(y(s_1))}\,ds.
\end{align}
Using equation \eqref{eq:psi3} we rewrite
\begin{align}
0 & = \frac{y_3(s_2)-y_3(s_1)}{\eta_1^\lambda(y(s_1))} \\
& = \int_{s_1}^{s_2} y_1(s)^2 v_y(s) \frac{\eta_1^\lambda(y(s))} {\eta_1^\lambda(y(s_1))}\,ds
+ \sum_{i=3}^n \int_{s_1}^{s_2} y_1(s)y_i(s)v_y(s) \frac{\eta_i^\lambda(y(s))} {\eta_1^\lambda(y(s_1))}\,ds.
\end{align}

Then, we can rewrite
    \begin{equation}\label{eq:I21new}
\mathcal{I}_2= \int_{s_1}^{s_2} y_1(s)^2 v_y(s) \frac{\eta_1^\lambda(y(s_1))-\eta_1^\lambda(y(s))} {\eta_1^\lambda(y(s_1))}\,ds - \sum_{i=3}^n \int_{s_1}^{s_2} y_1(s)y_i(s)v_y(s) \frac{\eta_i^\lambda(y(s))} {\eta_1^\lambda(y(s_1))}\,ds.
\end{equation}

To estimate the sum term of \eqref{eq:I21new}, recall that for $i \ge 3$ the functions $y_i$ satisfy $\dot y_i = v_yy_1\psi_i^\lambda(y)$. Therefore, since $y_i(s_1)=0$, we have
\begin{equation} 
    |y_i(s)| \le C_2 \sigma^{1/2} \left(\int_{s_1}^{s_2} y_1(r)^2\,dr\right)^{1/2}, \quad\forall i\geq 3.
\end{equation}

Then
\begin{equation}
    \left| \int_{s_1}^{s_2} y_1(s)y_i(s)v_y(s) \frac{\eta_i^\lambda(y(s))} {\eta_1^\lambda(y(s_1))}\,ds \right| \le \frac{C_2C_3\sigma}{m} \int_{s_1}^{s_2} y_1(s)^2\,ds.
\end{equation}

For the first term in \eqref{eq:I21new}, notice that  $|\eta_1^\lambda(y(s_1))|\ge m$.  Since $\rho \le \rho_1$, and $y$ goes through a point of $\bar{C}(\varepsilon')$, it follows by Remark \ref{rk:tau12} that $y$ is entirely contained in $W$, and the bound $|\dot y(s)| \le C_1$ holds for almost all $s$. Moreover, since $s-s_1\le\sigma\le2\rho_2$ and  $|\dot y|\le C_1$ we have $|y(s_1)-y(s)|\le 2C_1\rho_2$, therefore $|\eta_1^\lambda(y(s_1))-\eta_1^\lambda(y(s))|\le m/4$ by Remark \ref{rk:tau12}. We can estimate 
\begin{equation}
    \left| \int_{s_1}^{s_2} y_1(s)^2 v_y(s) \frac{\eta_1^\lambda(y(s_1))-\eta_1^\lambda(y(s))} {\eta_1^\lambda(y(s_1))}\,ds \right| \le \frac14 \int_{s_1}^{s_2} y_1(s)^2\,ds,
\end{equation}
hence
\begin{equation}
    \mathcal{I}_2 \le \left(\frac14 + \frac{(n-2) C_2C_3\sigma}{m}\right) \int_{s_1}^{s_2} y_1(s)^2\,ds.
\end{equation}

Collecting all the terms, we get the bound
\begin{equation}
    \int_{s_1}^{s_2} y_1(s)^2\,ds \le \beta^2\alpha + \left(\frac14 + \frac{(n-2)C_2C_3\sigma}{m}\right) \int_{s_1}^{s_2} y_1(s)^2\,ds.
\end{equation}

Since $\sigma \le 2\rho_3 \le \frac{m}{4(n-2)C_2C_3}$, we have
\begin{equation}
    \int_{s_1}^{s_2} y_1(s)^2\,ds \le \beta^2\alpha + \frac12 \int_{s_1}^{s_2} y_1(s)^2\,ds,
\end{equation}
from which it follows that
\begin{equation}
    \int_{s_1}^{s_2} y_1(s)^2\,ds \le 2\beta^2\alpha.
\end{equation}

Let us now establish a lower bound on $\int_{s_1}^{s_2} y_1(s)^2\,ds $. 
Let $\bar s \in [s_1,s_2]$ be such that $|y_1(\bar s)|=\beta$. Since $y_1(s_1)= y_1(s_2)=0$, and $|\dot y_1 (s)|\le C_1$ for a.e.\ $s\in [s_1,s_2]$, we have $\bar s - s_1 \ge \beta/C_1$ and $s_2 - \bar s \ge \beta/C_1$. Hence the intervals $I_1=[\bar s-\beta/C_1,\bar s]$ and $I_2=[\bar s,\bar s+\beta/C_1]$ are entirely contained in $[s_1,s_2]$. On each of these intervals $I_j$, $|y_1(s)|$ is bounded below by the linear function $\lambda_j$ which is equal to $\beta$ at $\bar s$ and to zero at the other endpoint. With a direct computation we have then
\begin{equation}
    \int_{s_1}^{s_2} y_1(s)^2\,ds \ge \int_{I_1} \lambda_1^2 + \int_{I_2} \lambda_2^2 = \frac{2\beta^3}{3C_1}.
\end{equation}  
Combining the upper and lower bounds on $\int_{s_1}^{s_2} y_1(s)^2\,ds$ we obtain
\begin{equation}
    \frac{2\beta^3}{3C_1} \le 2\beta^2\alpha.
\end{equation}
Therefore \eqref{eq:upper_bound_beta} holds true.

Finally, recalling \eqref{eq:2rho}, we have
\begin{equation}
    2\rho \le \sigma - \alpha + \beta\sigma C_2 \le \sigma - \alpha + 3C_1C_2\sigma\alpha = \sigma + (3C_1C_2\sigma - 1)\alpha.
\end{equation}

Since $\sigma \le 2\rho \le 2 \rho_4 \le 2/(7C_1C_2)$, we have $3C_1C_2\sigma - 1<0$. Recall that $\alpha \geq 0$ by definition. If $\alpha \neq 0$, the above computations would imply that, as soon as $\sigma \le 2\rho$, then $2\rho<\sigma$, which is impossible. Thus $\alpha=0$, i.e. $v_y(s) = 1$ almost everywhere. But then $u_y(s) = 0$ almost everywhere, and $y(s)= x^\lambda( s-s_1-\rho)$. This concludes the proof.$\qed$
\end{subsection}

\end{subsection}

\bibliographystyle{abbrvnat}
\bibliography{bibliography}
\end{document}